\documentclass[a4paper, 12pt]{article}
\pdfoutput=1 \usepackage{lineno}
\usepackage{subcaption}
\usepackage[utf8]{inputenc}
\usepackage[T1]{fontenc}
\usepackage[a4paper, margin=2.5cm]{geometry}
\usepackage{libertine} \usepackage{inconsolata} \usepackage{amsmath, amsthm, amssymb}
\usepackage{graphicx}
\usepackage[shortlabels]{enumitem}
\usepackage{authblk}
\usepackage{thmtools}
\usepackage{thm-restate}
\usepackage{authblk}
\usepackage[hyphens]{xurl}
\usepackage[colorlinks=true, linkcolor = bordeaux, citecolor = darkblue, urlcolor = darkblue]{hyperref}
\usepackage{float}
\usepackage{xcolor}
\usepackage{comment}

\def\longbox#1{\parbox{0.85\textwidth}{\textit{#1}}}
\def\stmt#1{\vspace{0cm}\begin{equation}\longbox{#1}\end{equation}}

\declaretheorem[name=Lemma, numberwithin = section]{lemma}
\declaretheorem[name=Theorem,sibling = lemma]{theorem}
\declaretheorem[name=Proposition, sibling=lemma]{proposition}
\declaretheorem[name=Observation, sibling=lemma]{obs}

\declaretheorem[name=Corollary, sibling=lemma]{corr}

\declaretheorem[name=Conjecture, sibling=lemma]{conjecture}
\declaretheorem[name=Problem, sibling=lemma]{problem}

\declaretheorem[name=Claim]{claim}

\usepackage[noabbrev,capitalise,nameinlink]{cleveref}
\crefname{claim}{Claim}{Claims}
\crefname{lemma}{Lemma}{Lemmas}
\crefname{theorem}{Theorem}{Theorems}
\crefname{proposition}{Proposition}{Propositions}
\crefname{question}{Question}{Questions}
\crefname{definition}{Definition}{Definitions}
\crefname{conjecture}{Conjecture}{Conjectures}
\crefname{obs}{Observation}{Observations}
\crefname{corr}{Corollary}{Corollaries}
\crefformat{equation}{(#2#1#3)}
\Crefformat{equation}{Equation #2(#1)#3}
\crefname{remark}{Remark}{Remarks}
\crefname{scenario}{Scenario}{Scenarios}
\crefname{problem}{Problem}{Problems}

\def\cqedsymbol{\ifmmode$\lrcorner$\else{\unskip\nobreak\hfil
\penalty50\hskip1em\null\nobreak\hfil$\lrcorner$
\parfillskip=0pt\finalhyphendemerits=0\endgraf}\fi}

\makeatletter
\newcommand{\leqnomode}{\tagsleft@true}

\newcommand{\reqnomode}{\tagsleft@false}
\makeatother

 \usepackage{tikz}
\usepackage{needspace}
\usepackage{soul}

\newcommand{\F}{\mathcal{F}}
\newcommand{\C}{\mathcal{C}}
\newcommand{\G}{\mathcal{G}}

\newtheoremstyle{named}{}{}{\itshape}{}{\bfseries}{.}{.5em}{\thmnote{#3}#1}
\theoremstyle{named}

\newtheoremstyle{named}{}{}{\itshape}{}{\bfseries}{.}{.5em}{\thmnote{#3}#1}
\theoremstyle{named}

\newenvironment{subproof}[1][\proofname]{\begin{proof}[#1]}{\end{proof}}

\newcommand{\say}[1]{``#1''} \newcommand\dd{\hbox{-}} 

\newcommand{\T}{\mathcal{T}}
\newcommand\abs[1]{\lvert#1\rvert}

\setlist[itemize]{topsep=0ex,itemsep=0ex,parsep=0.25ex}
\setlist[enumerate]{topsep=0ex,itemsep=0ex,parsep=0.25ex}

\definecolor{CornflowerBlue}{rgb}{0.39, 0.58, 0.93}
\definecolor{DarkGoldenrod}{rgb}{0.72, 0.53, 0.04}
\definecolor{BritishRacingGreen}{rgb}{0.0, 0.26, 0.15}
\definecolor{DarkMagenta}{rgb}{0.55, 0.0, 0.55}
\definecolor{VeryDarkMagenta}{rgb}{0.28, 0.0, 0.28}
\definecolor{AO}{rgb}{0.0, 0.5, 0.0}
\definecolor{BostonUniversityRed}{rgb}{0.8, 0.0, 0.0}
\definecolor{myRed}{rgb}{0.8, 0.0, 0.0}
\definecolor{DarkMidnightBlue}{rgb}{0.0, 0.2, 0.4}
\definecolor{DarkTangerine}{rgb}{1.0, 0.66, 0.07}
\definecolor{AppleGreen}{rgb}{0.55, 0.71, 0.0}
\definecolor{BrightUbe}{rgb}{0.82, 0.62, 0.91}
\definecolor{Amethyst}{rgb}{0.6, 0.4, 0.8}
\definecolor{DarkGray}{rgb}{0.52, 0.52, 0.51}
\definecolor{Gray}{rgb}{0.66, 0.66, 0.66}
\definecolor{BananaYellow}{rgb}{1.0, 0.88, 0.21}
\definecolor{Amber}{rgb}{1.0, 0.75, 0.0}
\definecolor{LightGray}{rgb}{0.83, 0.83, 0.83}
\definecolor{PrOrange}{rgb}{1.0, 0.56, 0.0}
\definecolor{DeepCarrotOrange}{rgb}{0.91, 0.41, 0.17}
\definecolor{cobalt}{RGB}{0,71,171}
\definecolor{brightpink}{RGB}{255,0,204} 
\definecolor{bordeaux}{RGB}{100,0,50}
\definecolor{olivegreen}{RGB}{107, 142, 35}
\definecolor{forestgreen}{RGB}{34, 139, 34}
\definecolor{darkgreen}{RGB}{0, 100, 0}
\definecolor{darkblue}{RGB}{25, 25, 112}
\definecolor{grun}{rgb}{0, 0.5, 0.5}
\definecolor{violet}{RGB}{177, 0.5, 255}

\usetikzlibrary{calc}
\usetikzlibrary{shapes,quotes}
\usetikzlibrary{backgrounds}
\usetikzlibrary{decorations.pathmorphing}
\usetikzlibrary {arrows.meta}

\title{Reuniting $\chi$-boundedness with polynomial~$\chi$-boundedness}

\author[1]{Maria Chudnovsky\thanks{Supported by NSF-EPSRC Grant DMS-2120644 and by AFOSR grant FA9550-22-1-0083.}}
\author[2]{Linda Cook\thanks{Supported by the Institute for Basic Science (IBS-R029-C1).}\thanks{Supported by the Gravitation programme NETWORKS (NWO grant no. 024.002.003) of the Dutch Ministry of Education, Culture and Science (OCW) and a Marie Skłodowska-Curie Action of the European Commission (COFUND grant no. 945045).
}}
\author[3]{James Davies}
\author[4]{Sang-il Oum\textsuperscript{\textdagger}}
\affil[1]{Department of Mathematics, Princeton University, Princeton, USA}
\affil[2]{Korteweg-de Vries Institute for Mathematics, University of Amsterdam, Amsterdam, the Netherlands}
\affil[3]{Department of Pure Mathematics and Mathematical Statistics, University of Cambridge, Cambridge, UK}
\affil[4]{Discrete Mathematics Group, Institute for Basic Science (IBS), Daejeon, South~Korea}
\affil[ ]{\small \textit{Email addresses:} \texttt{mchudnov@math.princeton.edu},
\texttt{l.j.cook@uva.nl},
\texttt{jgd37@cam.ac.uk},
\texttt{sangil@ibs.re.kr}}

\date{January 16, 2026}

\begin{document}
% \linenumbers
\maketitle

\begin{abstract}
    A class $\mathcal F$ of graphs is $\chi$-bounded if there is a function $f$ such that $\chi(H)\le f(\omega(H))$ for all induced subgraphs $H$ of a graph in $\mathcal F$. If $f$ can be chosen to be a polynomial, we say that $\mathcal F$ is polynomially $\chi$-bounded.
    Esperet proposed a conjecture that every $\chi$-bounded class of graphs is polynomially $\chi$-bounded. This conjecture has been disproved; it has been shown that there are classes of graphs that are $\chi$-bounded but not polynomially $\chi$-bounded.
    Nevertheless, inspired by Esperet's conjecture,
    we introduce Pollyanna classes of graphs. A class $\mathcal C$ of graphs is Pollyanna if $\mathcal C\cap \mathcal F$ is polynomially $\chi$-bounded for every 
$\chi$-bounded class $\mathcal F$ of graphs. 
    We prove that several classes of graphs are Pollyanna and also present some proper classes of graphs that are not Pollyanna.
\end{abstract}

\section{Introduction}
The \emph{chromatic number} of a graph~$G$, denoted by $\chi(G)$, is the minimum number of colors needed to color the vertices of $G$ such that adjacent vertices always receive distinct colors. 
A \emph{clique} of a graph is a set of pairwise adjacent vertices. 
We write $\omega(G)$ to denote the maximum size of a clique in a graph $G$, called the \emph{clique number} of~$G$.
For a graph $H$, we say $G$ is \emph{$H$-free} if $G$ has no induced subgraph isomorphic to~$H$. 

Obviously $\chi(G)\ge \omega(G)$.
In general, $\chi(G)$ is not bounded from above by any function of $\omega(G)$; there are constructions for triangle-free graphs with arbitrary large~$\chi(G)$ \cite{descartes1947three,Descartes1954, Mycielski1955, zykov1949}.
The strong perfect graph theorem~\cite{CRST2006} states that $\chi(H)=\omega(H)$ for all induced subgraphs $H$ of a graph $G$ if and only if $G$ has no odd cycles or their complements as an induced subgraph.
Such graphs are called perfect. 

Motivated by perfect graphs, Gy\'arf\'as~\cite{Gyarfas1975} initiated the study of graph classes on which $\chi(G)$ is bounded from above by a function of $\omega(G)$.
A class $\mathcal F$ of graphs is \emph{$\chi$-bounded} if there exists a function~$f$ such that $\chi(H)\le f(\omega(H))$ for all induced subgraphs~$H$ of a graph in~$\mathcal F$.\footnote{Sometimes, unlike Gy\'arf\'as~~\cite{Gyarfas1975}, some papers define $\chi$-boundedness to mean that $\chi(G)\le f(\omega(G))$ for every graph $G$ in the class and then use such a definition only for hereditary classes of graphs. For us, it is necessary to deal with $\chi$-bounded classes of graphs that are not necessarily hereditary, as when we define Pollyanna, we allow taking the intersection of our class with any $\chi$-bounded class, which is not necessarily hereditary.}

Such a function~$f$ is called a \emph{$\chi$-bounding function} for $\mathcal F$.
It is a well-known result of Erd\H{o}s that for every $g \geq 3$ there exist graphs arbitrarily large chromatic number and with no cycle of length less than $g$.
Hence, if $H$ contains a cycle, then the class of $H$-free graphs is not $\chi$-bounded.
(The converse is the well-known Gy\'arf\'as-Sumner conjecture \cite{Gyarfas1975, sumnerConjecture}).

A class of graphs is \emph{polynomially $\chi$-bounded} if it has a polynomial $\chi$-bounding function.
Examples of polynomially $\chi$-bounded classes of graphs include, perfect graphs \cite{CRST2006}, even-hole-free graphs \cite{CHUDNOVSKY2023evenhole}, circle graphs \cite{davies2021circle,davies2022improved}, rectangle intersection graphs \cite{asplund1960,chalermsook2021}, bounded twin-width graphs \cite{bourneuf2023bounded}, and $H$-free graphs for certain small forests $H$ \cite{scott2022stars,scott2022doublestar,CHUDNOVSKY2023fourpath}.
Note that for every graph~$H$, if the class of $H$-free graphs is polynomially $\chi$-bounded, then $H$ satisfies the celebrated Erd\H{o}s-Hajnal conjecture \cite{ERDOS198937}, which is largely open (see also \cite{EHsurvey}). A major open problem is whether the class of $P_5$-free graphs is polynomially $\chi$-bounded, since this would imply the smallest open case of the Erd\H{o}s-Hajnal conjecture. The best known $\chi$-bounding function for $P_5$-free graphs is quasi-polynomial \cite{scott2021polynomial}.

Esperet~\cite{Esperet2017} conjectured that every $\chi$-bounded class of graphs is polynomially $\chi$-bounded.
Recently, this conjecture was disproved by Bria\'nski, Davies, and Walczak~\cite{BDW2022} by extending ideas from a paper of Carbonero, Hompe, Moore, and Spirkl \cite{CARBONERO202363}.
In particular, Bria\'nski, Davies, and Walczak constructed classes of graphs that are $\chi$-bounded but not polynomially $\chi$-bounded.
Nevertheless, inspired by Esperet's conjecture, we consider its analog for proper classes of graphs.
We say that a class $\mathcal C$ of graphs is \emph{Pollyanna} if $\mathcal C\cap\mathcal F$ is polynomially $\chi$-bounded for every $\chi$-bounded class~$\mathcal F$ of graphs.
In other words, a class~$\mathcal C$ is Pollyanna if and only if the conjecture of Esperet holds inside~$\mathcal C$.
Note that every polynomially $\chi$-bounded class of graphs is Pollyanna, so Pollyanna classes of graphs generalize polynomially $\chi$-bounded classes.

\begin{figure}
    \begin{subfigure}{.30\textwidth}
        \centering
        \tikzstyle{v}=[circle, draw, solid, fill=black, inner sep=0pt, minimum width=3pt]
        \begin{tikzpicture}
            \foreach \x in {0,1,2,3,4,5,6} 
            { 
                \node [v] at (360*\x/7:1) (v\x) {};
            }
            \foreach \x in {0,1,2,3,4,5}
            {
                \pgfmathtruncatemacro{\lb}{\x+1}
                \foreach \y in {\lb,...,6}
                {
                    \draw (v\x)--(v\y);
                }
            }
            \foreach \y in {-2,-1,0,1,2} { 
                \draw (v0)--+(20*\y:1) node[v]{} ;
            }
        \end{tikzpicture}
        \caption{A $(7,5)$-pineapple.}\label{fig:pineapple}
    \end{subfigure}
    \begin{subfigure}{.30\textwidth}
        \centering
        \tikzstyle{v}=[circle, draw, solid, fill=black, inner sep=0pt, minimum width=3pt]
        \begin{tikzpicture}
            \foreach \x in {0,1,2,3,4} 
            { 
                \node [v] at (72*\x:1) (v\x) {};
            }
            \foreach \x in {0,1,2,3}
            {
                \pgfmathtruncatemacro{\lb}{\x+1}
                \foreach \y in {\lb,...,4}
                {
                    \draw (v\x)--(v\y);
                }
            }
            \draw (v0)--+(1,0) node[v]{} --+(2,0) node[v]{};
        \end{tikzpicture}
        \caption{A $5$-lollipop.}\label{fig:lollipop}
    \end{subfigure}
    \begin{subfigure}{.19\textwidth}
        \centering
        \tikzstyle{v}=[circle, draw, solid, fill=black, inner sep=0pt, minimum width=3pt]
        \begin{tikzpicture}
            \node [v] at (0,0) (v0){};
            \foreach \x in {-1,1}
            { 
                \node [v] at (30*\x:1) (v\x) {};
                \node [v] at (180+30*\x:1) (w\x) {}; 
                \draw (v0)--(v\x);
                \draw (v0)--(w\x);           
            }
            \draw (v-1)--(v1);
            \draw (w-1)--(w1);
        \end{tikzpicture}
        \caption{A bowtie.}\label{fig:bowtie}
    \end{subfigure}
    \begin{subfigure}{.19\textwidth}
        \centering
        \tikzstyle{v}=[circle, draw, solid, fill=black, inner sep=0pt, minimum width=3pt]
        \begin{tikzpicture}
            \node [v] at (0,0) (v0){};
            \node [v] at (60:1) (v1){};
            \node [v] at (120:1) (v2){};
            \draw (v1)--+(0,.5) node [v]{};
            \draw (v2)--+(0,.5) node [v]{};
            \draw (v0)--(v1)--(v2)--(v0);
        \end{tikzpicture}
        \caption{A bull.}\label{fig:bull}
    \end{subfigure}
    \caption{Forbidding any of these graphs makes a Pollyanna class of graphs.}\label{fig:smallgraph}
\end{figure}
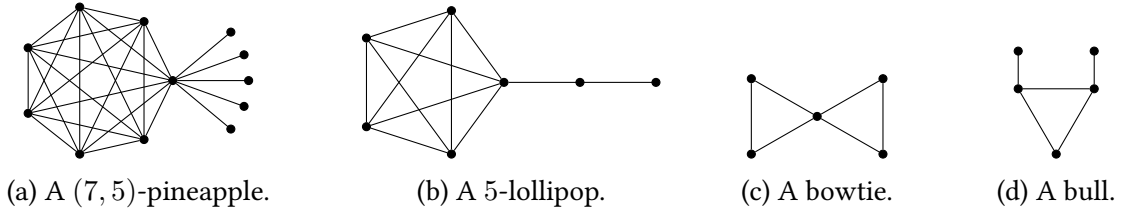

Here is our first main theorem. See \cref{fig:smallgraph} for an illustration of forbidden graphs; precise definitions are given in each corresponding section.
\begin{theorem}\label{thm:main}
Let $m$, $k$, $t$ be positive integers. 
The following graph classes are all Pollyanna.
    \begin{enumerate}[label=(\roman*)]
        \item The class of $mK_t$-free graphs.
        \item The class of $(t,k)$-pineapple-free graphs.
        \item The class of $t$-lollipop-free graphs.
        \item The class of bowtie-free graphs.
        \item The class of bull-free graphs.
    \end{enumerate}
\end{theorem}

None of the classes mentioned in \cref{thm:main} are $\chi$-bounded, 
because if a graph $H$ contains a cycle, then $H$-free graphs contain all graphs of large girth and therefore the chromatic number of $H$-free graphs is not bounded by the theorem of Erd\H{o}s~\cite{Erdos1959}.

We will actually prove something stronger than the statement in \cref{thm:main}.
For an integer~$n$, let $\chi^{(n)}(G)$ be the maximum chromatic number of an induced subgraph~$H$ of~$G$ such that $\omega(H)\leq n$.
For an integer $n$, we say a class $\F$ of graphs is \emph{$n$-good} if it is hereditary and there is some constant $m$ such that every $G \in \F$ satisfies $\chi^{(n)}(G) \leq m$.
Note that $n$-goodness is a strictly weaker condition than $\chi$-boundedness \cite{CARBONERO202363,BDW2022,GIPSSTT2022}.
We say a class~$\C$ of graphs is \emph{$n$-strongly Pollyanna} if $\C \cap \F$ is polynomially $\chi$-bounded for every $n$-good class $\F$ of graphs.
By considering the subclass~$\F$ of $\C$ consisting of graphs~$G$ with $\chi^{(n)}(G)\le c$, 
we can observe that 
\begin{quote}
$\mathcal C$ is $n$-strongly Pollyanna if and only if 
there is a function~$f$ such that 
\[ \chi(G)\le f(\chi^{(n)}(G),\omega(G))\]  
for every graph~$G\in \C$, where $f(c,\cdot)$ is a polynomial for every constant~$c$.
\end{quote}
We say that $\C$ is \emph{strongly Pollyanna} if it is $n$-strongly Pollyanna for some integer~$n$.
Note that for each $n\le 1$, a class $\C$ of graphs is $n$-strongly Pollyanna if and only if 
it is polynomially $\chi$-bounded.
We will show the following:
\begin{theorem}\label{thm:technical main}
    Let $m$, $k$, $t$ be positive integers. 
    The following statements hold.
    \begin{enumerate}[label=(\roman*)]
        \item The class of $mK_t$-free graphs is $(t-1)$-strongly Pollyanna.
        \item The class of $(t,k)$-pineapple-free graphs is $(2t-4)$-strongly Pollyanna.
        \item  The class of $t$-lollipop-free graphs is $(3t-6)$-strongly Pollyanna.
        \item\label{item:bowtie} The class of bowtie-free graphs is $3$-strongly Pollyanna.
        \item \label{item:bullfree} The class of bull-free graphs is $2$-strongly Pollyanna. 
    \end{enumerate}
\end{theorem}

The most difficult case of \cref{thm:technical main} is showing that bull-free graphs are $2$-strongly Pollyanna.
Bull-free graphs are of particular interest because of their complex structure, which was characterized by Chudnovsky~\cite{Chudnovsky2012, Chudnovsky2012a}, and have been widely studied.
Chudnovsky and Safra~\cite{ErdosHajnalBulls} showed that the bull satisfies the celebrated Erd\H{o}s-Hajnal Conjecture.
Bull-free graphs also have strong algorithmic properties \cite{TTV2017, Cray, coloringBulls}.
By using results of Chudnovsky~\cite{Chudnovsky2012, Chudnovsky2012a},
Thomass\'e, Trotignon, and Vu\v{s}kovi\'c~\cite{TTV2017} showed that there is a function $f$ such that 
$ \text{every bull-free $G$ satisfies }\chi(G) \leq f(\chi^{(2)}(G), \omega(G))
$.
Note that their function $f$ is far from being polynomial in $\omega(G)$ and
\cref{thm:technical main} shows $f$ can be taken to be a polynomial in $\omega(G)$.

\begin{figure}
    \centering
    \begin{tikzpicture}
        \tikzstyle{v}=[circle, draw, solid, fill=black, inner sep=0pt, minimum width=3pt]
        \tikzstyle{w}=[circle, draw=red, thick,solid, fill=white, inner sep=0pt, minimum width=3pt]
        \foreach \x in {1,2,3,4,5} {
            \draw (72*\x+90:.75) node [v] (v\x) {}
            --(72*\x+90:1.5) node [w] (w\x) {};
        }
        \draw (v1)--(v2)--(v3)--(v4)--(v5)--(v1);
        \draw (v1)--(v3)--(v5)--(v2)--(v4)--(v1);
        \begin{scope}[on background layer]
\draw [line width=.5cm,draw=gray,opacity=.2,line cap=round] (w1)--(w2)--(w3)--(w4)--(w5)--(w1);
        \end{scope}
        \node at (0,-2) {pentagram spider};
        \begin{scope}[xshift=5cm]
            \draw (90:.5) node[v] (x1){};
            \draw (90+120:.5)node[v] (x2){};
            \draw (90-120:.5)node[v] (x3){};
            \draw (x1)--(x2)--(x3)--(x1);
            \draw (x1) -- +(45:1) node[w] (x11){} -- +(90:1.414214) node[w] (x12) {}
            -- +(180-45:1)node[w](x13){}--(x1);
            \draw (x2) -- +(120+45:1) node[w] (x21){} -- +(120+90:1.414214) node[w] (x22) {}
            -- +(120+180-45:1)node[w](x23){}--(x2);
            \draw (x3) -- +(-120+45:1) node[w] (x31){} -- +(-120+90:1.414214) node[w] (x32) {}
            -- +(-120+180-45:1)node[w](x33){}--(x3);
            \foreach \x in {1,2,3}{
                \draw (x\x 1)--(x\x 3);
                \draw (x\x 2)--(x\x);
            }
            \begin{scope}[on background layer]
\draw [line width=.8cm,draw=gray,opacity=.2,line cap=round] (x12)--(x13)--(x21)--(x22)--(x23)--(x31)--(x32)--(x33)--(x11)--(x12);
            \end{scope}
            \node at (0,-2) {tall strider};
        \end{scope}
        \begin{scope}[xshift=10cm]
            \draw (0,0) node[v] (y){};
            \draw (90:.5) node[v] (y1){};
            \draw (90+120:.5)node[v] (y2){};
            \draw (90-120:.5)node[v] (y3){};
            \draw (y1)--(y2)--(y3)--(y1);
            \draw (y1) -- +(60:1) node[w] (y11){} 
            -- +(120:1)node[w](y13){}--(y1);
            \draw (y2) -- +(120+60:1) node[w] (y21){} 
            -- +(120+120:1)node[w](y23){}--(y2);
            \draw (y3) -- +(-120+60:1) node[w] (y31){} 
            -- +(-120+120:1)node[w](y33){}--(y3);
            \foreach \x in {1,2,3}{
                \draw (y)--(y\x);
            }
            \begin{scope}[on background layer]
\draw [line width=.8cm,draw=gray,opacity=.2,line cap=round] (y11)--(y13)--(y21)--(y23)--(y31)--(y33)--(y11);
            \end{scope}
            \node at (0,-2) {short strider};
        \end{scope}
    \end{tikzpicture}
\caption{A pentagram spider, a tall strider, and a short strider are graphs obtained from the above figure by adding any additional edges between two red hollow vertices.
}
    \label{fig:willow-forbids}
\end{figure}
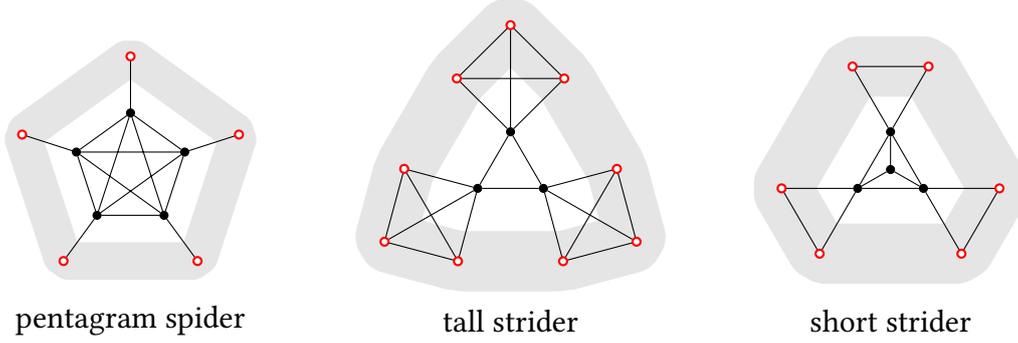
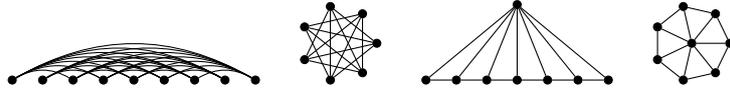
\begin{figure}
    \centering
    \begin{tikzpicture}
        \tikzstyle{v}=[circle, draw, solid, fill=black, inner sep=0pt, minimum width=3pt]
        \foreach \i in {1,2,...,9} {
            \node at (0.4*\i,0) [v](v\i) {};
        }
        
        \foreach \x in {1,2,...,7}
        {
            \pgfmathtruncatemacro{\lb}{\x+2}
            \foreach \y in {\lb,...,9}
            {
                \draw [bend left] (v\x) to(v\y);
            }
        }
    \end{tikzpicture}
    \quad
    \begin{tikzpicture}
        \tikzstyle{v}=[circle, draw, solid, fill=black, inner sep=0pt, minimum width=3pt]
        \foreach \i in {1,2,...,7} {
            \node at (\i*360/7:.5) [v](v\i) {};
        }
        \foreach \y in {3,...,6}
        {
                \draw (v1)--(v\y);
        }
        \foreach \x in {2,3,4,5}
        {
            \pgfmathtruncatemacro{\lb}{\x+2}
            \foreach \y in {\lb,...,7}
            {
                \draw (v\x)--(v\y);
            }
        }
    \end{tikzpicture}    
    \quad
    \begin{tikzpicture}
        \tikzstyle{v}=[circle, draw, solid, fill=black, inner sep=0pt, minimum width=3pt]
        \node at (0,0) [v](c) {};
        \foreach \i in {1,2,...,7} {
            \node at (\i*0.4-4*0.4,-1) [v](v\i) {};
            \draw (c)--(v\i);
        }
        \foreach \x in {1,2,...,6}
        {
            \pgfmathtruncatemacro{\lb}{\x+1}
            \draw (v\x)--(v\lb);
        }
    \end{tikzpicture}
    \quad
    \begin{tikzpicture}
        \tikzstyle{v}=[circle, draw, solid, fill=black, inner sep=0pt, minimum width=3pt]
        \node at (0,0) [v](c) {};
        \foreach \i in {1,2,...,7} {
            \node at (\i*360/7:.5) [v](v\i) {};
            \draw (c)--(v\i);
        }
        \foreach \x in {1,2,...,6}
        {
            \pgfmathtruncatemacro{\lb}{\x+1}
            \draw (v\x)--(v\lb);
        }
        \draw (v7)--(v1);
    \end{tikzpicture}
    \caption{Graphs $\overline{P_9}$, $\overline{C_7}$, $F_7$, and $W_7$. The class of ($\overline{P_9}$, $\overline{C_7}$, $F_7$, $W_7$)-free graphs is not Pollyanna.}\label{fig:nonwillows}

\end{figure}

Our second main theorem shows that a certain proper class of graphs is not Pollyanna, which generalizes the theorem of Bria\'nski, Davies, and Walczak~\cite{BDW2022} that the class of all graphs is not Pollyanna.
See \cref{fig:willow-forbids,fig:nonwillows} for an illustration of pentagram spiders, tall striders, short striders, $F_7$, $W_7$, the complement $\overline{P_9}$ of $P_9$, and the complement $\overline{C_7}$ of $C_7$; precise definitions are given in \cref{sec:willows}.

\begin{theorem}
    \label{thm:nonpollyana}
    Let $\mathcal F$ be 
    the set of all pentagram spiders, all tall striders, all short striders, 
    $\overline{P_{9}}$, 
    $\overline{C_n}$,
    $F_n$, and $W_n$
    for all $n\ge 7$.
    Then the class of $\F$-free graphs is not Pollyanna.
\end{theorem}

We will actually prove something significantly more general than \cref{thm:nonpollyana} (see \cref{non-pollyanna,non-nearchi}), where $\F$ can be any finite collection of graphs that are not willows. We will introduce willows in \cref{sec:nonpollyanna}.

The paper is organized as follows.
\cref{sec:prelim} reviews basic definitions and properties.
\cref{sec:mKt,sec:pineapple,sec:lollipop,sec:bowtie,sec:bull} each deal with the proof of a different case of \cref{thm:main} in order, and we remark that each of these sections can be read independently of each other.
\cref{sec:willows,sec:nonpollyanna} deal with the proof of \cref{thm:nonpollyana}.
\cref{sec:conclusions} ends the paper with a discussion of further work and several open problems.

\section{Preliminaries}
\label{sec:prelim}
We denote the complement of a graph $G$ by $\overline{G}$.
For a graph $H$, a graph $G$ is \emph{$H$-free} if $G$ has no induced subgraph isomorphic to $H$.
For a set $\F$ of graphs, a graph $G$ is \emph{$\F$-free} if $G$ is $H$-free for every $H\in\F$.
For a vertex $v$ of a graph $G$, we write $N_G(v)$ to denote the set of all neighbors of $v$.
For a set $S \subseteq V(G)$, we will denote $\cup_{s \in S} N_G(s) \setminus S$ by $N(S)$.
In situations where it is not ambiguous, we will denote $N_G(v)$ by $N(v)$ and $N_G(S)$ by $N(S)$.
For two disjoint sets $A$ and~$B$ of vertices, we say that $A$ is \emph{anti-complete} to $B$ if there are no edges between $A$ and~$B$, and \emph{complete} to $B$ if every vertex in~$A$ is adjacent to every vertex in~$B$. 
If $A$ is neither complete nor anti-complete to $B$, then we say $A$ is \emph{mixed} on $B$.
We let $P_t$ denote the path on $t$ vertices.
The length of a path or a cycle is the number of its edges.
For $S, T \subseteq V(G)$ the distance between~$S$ and $T$ is the length of a shortest path with one end in $S$ and the other end in~$T$.

In the rest of this section, we detail further preliminaries that we require to show that the class of $t$-lollipop-free and the class of bull-free graphs are Pollyanna.

A \emph{homogeneous set} of a graph $G$ is a set $X$ of vertices such that $1<\abs{X}<\abs{V(G)}$ and every vertex in $V(G)\setminus X$ is either complete or anti-complete to $X$.
\emph{Substituting} a vertex $v$ of a graph~$G$ by a graph $H$ is an operation that creates a graph obtained from the disjoint union of $H$ and $G-v$ by
adding an edge between every vertex of $H$ and every neighbor of $v$ in $G$.
Notice that if $|V(G)|,|V(H)|>1$, then $V(H)$ is a homogeneous set in this new graph.
We require a theorem of Chudnovsky, Penev, Scott, and Trotignon~\cite{CPST2013} that substitution preservers polynomial $\chi$-boundedness.
Given a class $\mathcal{C}$ of graphs, we let $\mathcal{C}^*$ denote the closure of $\mathcal{C}$ under substitutions and disjoint unions.
\begin{theorem}[Chudnovsky, Penev, Scott, and Trotignon~\cite{CPST2013}]\label{thm:substitution-orig}
    Let $\mathcal C$ be a class of graphs.
    If $\mathcal{C}$ is polynomially $\chi$-bounded, then so is $\mathcal{C}^*$.
\end{theorem}

We further require some results on perfect graphs.
A \emph{hole} is an induced cycle of length at least four.
The \emph{parity} of a hole (or path) is the parity of its length.
An induced subgraph $A$ of a graph $G$ is an \emph{antihole} if $V(A)$ induces a hole in $\overline{G}$.
A graph $G$ is called \emph{perfect} if every induced subgraph $H$ of $G$ satisfies $\omega(H) = \chi(H)$.
The \say{Strong Perfect Graph Theorem} of Chudnovsky, Robertson, Seymour, and Thomas~\cite{CRST2006} states that a graph is perfect if and only if it does not contain an odd hole or an odd antihole.

We do not require the full force of the strong perfect graph theorem and so, we will instead use the following three results. They are easy corollaries of the strong perfect graph theorem, but they were proven several years earlier and have much shorter proofs.

\begin{theorem}[Seinsche~\cite{Seinsche1974}] \label{cographs}
    Every $P_4$-free graph is perfect.
\end{theorem}

\begin{theorem}[Chv{\'a}tal and Sbihi~\cite{chvatal1987bull}] \label{bullperfect}
    A bull-free graph is perfect if and only if it does not contain an odd hole or odd antihole.
\end{theorem}

\begin{lemma}[Lov\'asz~\cite{Lovasz1972}]\label{lovaszreplacement}
    The class of perfect graphs is closed under taking substitutions. 
\end{lemma}

\section{Adding a clique}
\label{sec:mKt}

We write $H\cup F$ to denote the disjoint union of two graphs $H$ and $F$. We prove that if the class of $H$-free graphs is Pollyanna, then so is the class of $(K_t \cup H)$-free graphs.
Our proof is very similar to Wagon's proof~\cite{Wagon1980} that the class of $mK_2$-free graphs is polynomially $\chi$-bounded for each positive integer $m$.

\begin{proposition}\label{thm:K_t-union-H} 
    Let $t\ge 1$ be an integer.
    If the class of $H$-free graphs is Pollyanna, then 
    the class of $(K_{t}\cup H)$-free graphs is Pollyanna.
\end{proposition}
\begin{proof}
    Let $\mathcal C$ be the class of $(K_t\cup H)$-free graphs.
    Let $\mathcal D$ be the class of $H$-free graphs.
    Let $\mathcal F$ be a 
    $\chi$-bounded hereditary class of graphs with a $\chi$-bounding function~$f$.  
    We may assume that $f$ is an increasing function.
    Assume that $\mathcal F\cap \mathcal D$ is $\chi$-bounded by a $\chi$-bounding polynomial $g$. 
    We may also assume that $g$ is an increasing function.

    Let $G$ be a graph in $\mathcal F\cap \mathcal C$.
    To prove that $\mathcal F \cap \mathcal C$ is $\chi$-bounded, we claim that 
    \begin{equation}
        \label{eq:mKt}
        \chi(G)\le    
        \binom{\omega(G)}{t-1} f(t-1)
        +\binom{\omega(G)}{t}g(\omega(G)). 
    \end{equation}
    We may assume that $\omega(G)\ge t$ because otherwise $\chi(G)\le f(t-1)$.
    Let $K$ be a clique of $G$ with $\abs{K}=\omega(G)$.

    Now, for each subset $M$ of $K$ with $\abs{M}=t-1$, 
    let $A_M$ be the set of all vertices in $V(G)\setminus K$ that are complete to $K\setminus M$. 
    Since $K \setminus M$ is complete to $A_M$, we have that $\omega(G[A_M]) \le \omega(G)-\omega(G[K \setminus M]) = \omega(G) - (\omega(G)-(t-1)) = t-1$.
    Therefore, $\chi (G[A_M])\le f(\omega(G[A_M])) \le f(t-1)$.

    For each subset $N$ of $K$ with $\abs{N}=t$, 
    let $A'_{N}$ be the set of all vertices in $V(G)\setminus K$ 
    that are anti-complete to $N$. 
    Since $G$ has no induced subgraph isomorphic to $K_{t}\cup H$,     
    $G[A'_{N}]\in \mathcal D$.
    This implies that $\chi(G[A'_{N}])\le g(\omega(G))$.
    Observe that every vertex in $V(G)$ is in $M\cup A_M$ for some $M\subseteq K$ with $\abs{M}=t-1$, 
    or in $A'_{N}$ for some $N\subseteq K$ with $\abs{N}=t$.
    Thus we deduce that \eqref{eq:mKt} holds since there are $\binom{\omega(G)}{\omega(G) -(t-1)} = \binom{\omega(G)}{t-1}$ such choices for $M$, and $\binom{\omega(G)}{t}$ choices for~$N$.
\end{proof}
We can use the almost same proof to prove the following. 
\begin{proposition}
    If the class of $H$-free graphs is $(t-1)$-strongly Pollyanna, 
    then the class of $(K_t\cup H)$-free is $(t-1)$-strongly Pollyanna.
    \qed
\end{proposition}

Since the class of $K_t$-free graphs is trivially $(t-1)$-strongly Pollyanna, we deduce the following corollary.
\begin{corr}\label{mKt}
    The class of $mK_t$-free graphs is $(t-1)$-strongly Pollyanna.\qed 
\end{corr}

\cref{mKt} implies the aforementioned result of Wagon~\cite{Wagon1980} that the class of $mK_2$-free graphs is polynomially $\chi$-bounded for each positive integer $m$.

\section{Pineapple-free graphs}
\label{sec:pineapple}

For positive integers $t$ and $k$, a \emph{$(t,k)$-pineapple} is a graph 
obtained 
by attaching $k$ pendant edges to a vertex of a complete graph $K_t$, see \cref{fig:pineapple}.
In this section, we will show that the class of $(t,k)$-pineapple-free graphs is Pollyanna. First, we need to introduce Ramsey's theorem with some explicit bounds.

For positive integers $s$ and $t$, let $R(s,t)$ be the minimum positive integer $N$ such that every graph on $N$ vertices contains a clique of size $s$ or an independent of size $t$. Ramsey's theorem~\cite{Ramsey1930} states that $R(s,t)$ exists. 
Erd\H{o}s and Szekeres~\cite{ES1935} proved the following upper bound. 
\begin{proposition}[Erd\H{o}s and Szekeres~\cite{ES1935}]\label{prop:ramsey}
    For positive integers $s$ and $t$, we have 
    $R(s,t)\le \binom{s+t-2}{t-1}$.
\end{proposition}
Because of \cref{prop:ramsey}, if $t$ is a fixed constant, then $R(s,t)$ is bounded from above by a degree-$(t-1)$ polynomial in $s$.

We are now ready to prove that the class of pineapple-free graphs is Pollyanna.

\begin{proposition}\label{prop:pineapple}
    Let $t$, $k$ be positive integers.
The class of $(t,k)$-pineapple-free graphs is $(2t-4)$-strongly Pollyanna.
\end{proposition}
\begin{proof}
    We may assume that $t>2$, because otherwise the class of $(t,k)$-pineapple-free graphs is polynomially $\chi$-bounded by \cref{prop:ramsey}.
    Let $\F$ be a hereditary class of graphs 
    and let $C$ be a positive integer 
    such that $\chi(G)\le C$ whenever $G\in \mathcal F$ and $\omega(G)\le 2t-4$.
    Let $\G$ be the class of $(t,k)$-pineapple-free graphs.
    Let $G\in\mathcal F\cap \G$.  
Let 
    \[ m(x) =C\sum_{i=1}^{t-2} \binom{x}{i}, \quad
         g(x)=\left(t \binom{x}{t}+1\right)
        m(x)  
        \binom{x+k-2}{k-1}.           
    \] 
    Let $\omega$ be a positive integer.
    We claim that if $\omega(G)\le \omega$, then 
    $\chi(G)\le g(\omega)$.
    We proceed by induction on $\abs{V(G)}$.
    We may assume that $\omega(G)\ge 2t-3$ because otherwise $\chi(G)\le C\le g(\omega)$.

    Let $K$ be a clique of size $\omega(G)$.
    For a nonempty subset $M$ of $K$ with $\abs{M}<t-1$, 
    let $A_M$ be the set of vertices in $V(G)\setminus K$ that are complete to $K\setminus M$ and anti-complete to $M$.
    Then $\omega(G[A_M\cup M])= \abs{M}$ and therefore 
    $\chi(G[A_M\cup M])\le C$.
    Let $S$ be the union of all $A_M$ for every choice of $M \subseteq K$ satisfying $1 \leq \abs{M} < t-1$.
    Then, 
    \begin{equation}\label{eq:-Chi-S}
        \begin{aligned}
    \chi(G[K\cup S]) 
    &\leq \sum_{v\in K} \chi(G[A_{\{v\}}\cup \{v\}])
    + \sum_{M\subseteq K,~ 2\le \abs{M}<t-1} \chi(G[A_M])\\    
    &\leq C\sum_{i=1}^{t-2} \binom{\omega}{i} = m(\omega).
        \end{aligned}
    \end{equation}

    \begin{figure}
        \centering
        \begin{tikzpicture}
            \node (1) [fill=white,draw, rounded rectangle,minimum width=4cm] at (0,0) {A maximum clique $K$};
            \node (2) [fill=white,draw, rounded rectangle,minimum width=2cm] at (-1.5,-1) {$S$};
            \node (3) [fill=white,draw, rounded rectangle,minimum width=2cm] at (1.5,-1) {$T$};
            \node (4) [fill=white,draw, rounded rectangle,minimum width=5cm] at (0,-2) {Non-neighbors of $K$};
            \begin{scope}[on background layer]
                \draw [draw=red,fill=red,opacity=0.3,fill opacity=0.3](-2,0)--(-2,-1)--(1.5,0)--(-1.5,0);
                \draw [draw=blue,fill=blue,opacity=0.3,fill opacity=0.3](-1.5,0)--(2,-1)--(1.5,0)--(1.5,0);
                \draw [line width=5pt,opacity=.3,draw=gray,decorate,decoration=coil] (-1.5,-1)--(1.5,-1)--(1,-2);
                \draw [line width=5pt,opacity=.3,draw=gray,decorate,decoration=coil](-1,-2)--(-1.5,-1);

            \end{scope}
        \end{tikzpicture}
        \caption{An illustration for the proof of \cref{prop:pineapple}.}\label{fig:pineappleproof}
    \end{figure}
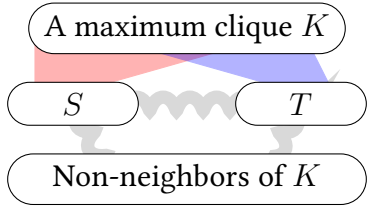
    For a subset~$N$ of $K$ with $\abs{N}=t-1$ and a vertex $v$ of $K\setminus N$, 
    let $A'_{N,v}$ be the set of vertices in $N(v)\setminus K$ that are anti-complete to $N$.
    Clearly, $\omega(A'_{N,v}) \le \omega-1$. 
    As $G$ is $(t,k)$-pineapple-free, $G[A'_{N,v}]$ has no independent set of size $k$. Thus, by Ramsey's theorem, $\abs{A'_{N,v}}< R(\omega,k)$.
    
Note that, by definition, every vertex $u \in N(K)$ with at least $t-1$ non-neighbors in $K$ is in $A'_{N, v}$
    for some $N \subseteq K \setminus N(u)$ and $v \in K$ with $\abs{N}=t-1$.
Let $T$ be the union of all $A'_{N,v}$ for every choice of $N \subseteq K$ and $v \in K \setminus N$ such that $\abs{N} = t-1$.
Then, 
    \begin{equation}\label{eq:size-of-T}
        \abs{T}<   t \binom{\omega}{t} R(\omega,k).
    \end{equation}
 
    It follows from the definition of $S$ and $T$ that $S$ is the set of all vertices in $N(K)$ with fewer than $t-1$ non-neighbors in $K$ and $T$ is the set of all vertices in $N(K)$ with at least $t-1$ non-neighbors in $K$, see \cref{fig:pineappleproof}.
    Hence, $N(K) = S \cup T$.
   
    Since $\abs{K}\ge 2t-3$, 
    each vertex $v\in S$ has at least $t-1$ neighbors in $K$ and therefore 
    $\abs{N(v)\setminus (K \cup N(K))} < R(\omega,k)$ because $G$ is $(t,k)$-pineapple-free.
    Then by \cref{eq:size-of-T}, each vertex $v\in K\cup S$ has fewer than   
    $\alpha:=\left(t \binom{\omega}{t}+1\right)R(\omega,k)$
    neighbors in $V(G)\setminus (K\cup S)$.
    Let $c_1:V(G\setminus (K\cup S))\to \{1,2,\ldots,g(\omega)\}$ be a coloring of $G\setminus (K\cup S)$ obtained by the induction hypothesis.
By \eqref{eq:-Chi-S}, there is a coloring $c_2:K\cup S\to\{1,2,\ldots,m(\omega)\}$ of $G[K\cup S]$.
    We define a coloring $c:V(G)\to\{1,2,\ldots,g(\omega)\}$ of $G$ as follows.
    For $v\in V(G\setminus (K\cup S))$, define $c(v):=c_1(v)$.
    Since every $v \in K\cup S$ has fewer than $\alpha$ neighbors in $V(G) \setminus (K \cup S)$, there is some choice of $c(v) \in \{\alpha(c_2(v)-1) +1, \alpha(c_2(v) -1) +2, \dots, \alpha c_2(v)\}$ that is not present in~$N(v) \setminus S$.
    Since $c_2$ was a proper coloring of $G[K\cup S]$, it follows that $c$ is a proper coloring for~$G$ with at most $\max(\alpha m(\omega),g(\omega))$ colors.
Note that $R(\omega,k)\le \binom{\omega+k-2}{k-1}$ by \cref{prop:ramsey}.
    This completes the proof.
\end{proof}

\section{Lollipop-free graphs}
\label{sec:lollipop}
Let $t \geq 1$ be a fixed integer.
The \emph{$t$-lollipop} is a graph obtained from the disjoint union of the complete graph $K_t$ on $t$ vertices and the path graph $P_2$ on $2$ vertices by adding an edge, see \cref{fig:lollipop}.
Note that a $t$-lollipop is a $(t,1)$-pineapple whose pendant edge is subdivided once.
In this section, we aim to show that the class of $t$-lollipop-free graphs is Pollyanna.

We say that a graph $H$ is \emph{tidy} if 
    $\abs{V(H)}\ge 2$ and 
    for any partition of $V(H)$ into two nonempty subsets $M$ and $N$,  one of the following holds.
    \begin{enumerate}[label=(U\arabic*)]
        \item \label{item:U1}
        $H[M]$ contains a clique $K$ of size $t-1$ 
        and $N$ has a vertex  anti-complete to $K$ in $H$.
\item \label{item:U2}
        $H[N]$ contains  a clique $K$ of size $t-1$ and $H$ has adjacent vertices $x\in M$ and $y\in N\setminus K$ such that both $x$ and $y$ are anti-complete to $K$ in $H$.
    \end{enumerate}

\begin{lemma}\label{2k2t-is-austere}
    Let $t\ge 3$ be an integer.
    The disjoint union of two copies of $K_{2t-3}$ is tidy. 
\end{lemma}
\begin{proof}
    Let $S_1$, $S_2$ be the two cliques of cardinality $2t-3$ and let $H$ be the disjoint union of $S_1$ and $S_2$. 
    Let $M$, $N$ be nonempty disjoint subsets of $V(H)$ such that $M\cup N=V(H)$.
    We may assume \ref{item:U1} does not hold for $M, N$.

    \begin{claim}\label{lollipop:N-big}
        For each $i\in\{1,2\}$, if $S_i \cap N \neq \emptyset$, then $\abs{S_{3-i}\cap N}\ge t-1$.
    \end{claim}
    \begin{subproof}
Since \ref{item:U1} does not hold for $S_{3-i}$, we deduce that $\abs{S_{3-i}\cap M}<t-1$.
    Therefore $\abs{S_{3-i}\cap N}\ge t-1$.
    \end{subproof}

    We may assume $S_1 \cap N \neq \emptyset$. By \cref{lollipop:N-big}, we obtain $|S_2 \cap N| \geq t-1$.
    Since $t \geq 2$, this implies $S_2 \cap N \neq \emptyset$ and therefore by \cref{lollipop:N-big}, we have $|S_1 \cap N| \geq t-1$.

    Let $x\in M$. Then $x\in S_i$ for some $i\in\{1,2\}$.
    By the previous paragraph, there is some $y \in S_{i} \cap N$ and
    some subset $K \subseteq S_{3-i} \cap N$ of cardinality $t-1$.
    Now, $K$, $x$, and $y$ satisfy  \ref{item:U2}.
\end{proof}

A set $S$ of vertices is a \emph{split} if it has the property that for every $v, u \notin S$ where $v$ is complete to $S$ and $u$ is mixed on $S$, the vertices $u$ and $v$ are adjacent.
A set $S$ of vertices of a graph $G$ is \emph{fair} if for every $v \in N(S)$, either $v$ is complete to $S$ 
    or $\omega(G[S \setminus N(v)]) \ge t-1$.

\begin{lemma}\label{austere-is-fair-split}
    Let $t\ge 3$ be an integer.
    If $G$ is a $t$-lollipop-free graph 
    and $G[S]$ is tidy for $S\subseteq V(G)$, 
    then $S$ is a fair split.
\end{lemma}
\begin{proof}
    Let us first show that $S$ is a split.
    Suppose that a vertex $v\in V(G)\setminus S$ is complete to~$S$, a vertex $u\in V(G)\setminus S$ is mixed on $S$, and $u$ is non-adjacent to~$v$.
    Let $N=N_G(u)\cap S$ and $M=S\setminus N$.
    As $M,N\neq\emptyset$, 
    \ref{item:U1} or \ref{item:U2} holds.
    If \ref{item:U1} holds with the clique $K\subseteq M$ and the vertex $w\in N$, 
    then $G[K\cup \{w,u,v\}]$ induces a $t$-lollipop.
    If \ref{item:U2} holds with the clique $K\subseteq N$ and two adjacent vertices $x\in M$, $y\in N$, 
    then $G[K\cup \{x,y,u\}]$ induces a $t$-lollipop. 
    This proves that $S$ is a split.

    Now let us show that $S$ is fair. Suppose that $v$ is not complete to $S$ and $\omega(G[S\setminus N(v)])<t-1$.
    Let $N=N(v)\cap S$ and $M=S\setminus N$.
    By the assumption on $\omega(G[S\setminus N(v)])$, \ref{item:U1} does not hold and therefore \ref{item:U2} holds with the clique $K\subseteq N$ and two adjacent vertices $x\in M$, $y\in N\setminus K$. 
    This implies that $G[K\cup \{x,y,v\}]$ induces a $t$-lollipop, a contradiction.
\end{proof}

The following lemma is an immediate consequence of \cref{2k2t-is-austere,austere-is-fair-split}.
 For brevity, we will denote the disjoint union of two copies of $K_{2t-3}$ by $2K_{2t-3}$.

\begin{lemma}\label{2k2t}
    Let $t\ge 3$ be an integer.
    Let $G$ be a $t$-lollipop-free graph and let $S \subseteq V(G)$ induce a copy of $2K_{2t-3}$.
    Then $S$ is a fair split.
    \qed
\end{lemma}

Next, we show that if some fair split is contained in the neighborhood of a vertex, then $G$ has a homogeneous set.

\begin{lemma}\label{fair-split-hom}
    Let $t\ge 3$ be an integer.
    Let $G$ be a $t$-lollipop-free graph and $v$ be a vertex.
    If some $S\subseteq N(v)$ is a fair split in~$G$, 
    then $G$ has a homogeneous set.
\end{lemma}
\begin{proof}
    Let $X$ be the set of all vertices in $V(G)\setminus S$ complete to $S$. As $v\in X$, the set $X$ is nonempty.
    Let $Y$ be the set of all vertices in $V(G)\setminus S$ mixed on $S$. 
    Since $S$ is a split, $X$ is complete to $Y$.
    
    Let $Z$ be the set of vertices in $V(G)\setminus (S\cup X\cup Y)$ that have a path to $S$ in $G\setminus X$.
    We claim that $Z$ is complete to $X$. 
    Suppose not. Then there are $x\in X$ and $z\in Z$ such that $x$ is non-adjacent to $z$. 
    Let $P$ be a path from $z$ to $S$ in $G\setminus X$. 
    We choose $x$, $z$, and $P$ such that the length of $P$ is minimized.
    By such a choice, $V(P)\setminus \{z\}$ is complete to $x$
    and $V(P)\cap Y$ has a unique vertex, say $y$.
    Because $S$ is fair, 
    $\omega(G[S\setminus N(y)])\ge t-1$.
    Let $K$ be a clique of size~$t-1$ in $G[S\setminus N(y)]$.
    Let $z'$ be the vertex on $P$ adjacent to $z$. 
    Then $z'$ is anti-complete to~$K$ so
    $G[K\cup \{x,z',z\}]$ is a $t$-lollipop, a contradiction.
    This proves that $Z$ is complete to $X$.
    Since $V(G)\setminus (S\cup X\cup Y\cup Z)$ is anti-complete to $S\cup Y\cup Z$ in~$G$, it follows that 
    $S\cup Y\cup Z$ is a homogeneous set in~$G$.
\end{proof}

Let $2K_{2t-3}^*$ be the graph obtained from $2K_{2t-3}$ by adding a new vertex adjacent to all other vertices.
Before showing that the class of $t$-lollipop-free graphs is Pollyanna, as an intermediate step, we first show that the class of {($t$-lollipop, $2K_{2t-3}^*$)}-free graphs is Pollyanna.
\begin{lemma}\label{lem:lollipop-sub}
    For every integer $t\ge 3$,
    the class of {($t$-lollipop, $2K_{2t-3}^*$)}-free graphs is {$(3t-6)$-strongly} Pollyanna.
\end{lemma}
\begin{proof}
Let $\mathcal C$ be the class of $t$-lollipop-free $2K_{2t-3}^*$-free graphs.
    Let $\mathcal F$ be a 
    hereditary class of graphs 
    and let $m$ be a positive integer
    such that $\chi(G)\le m$ whenever $G\in \mathcal F$ and $\omega(G)\le 3t-6$.

    Let $G$ be a graph in $\mathcal F\cap \mathcal C$.
    For every vertex $v$ of $G$, 
    $G[N(v)]$ has no induced subgraph isomorphic to $2K_{2t-3}$ because $G$ is $2K_{2t-3}^*$-free. 
    We may assume that $\omega(G)> 3t-6$ because otherwise $\chi(G)\le m$.
    Let $K$ be a clique of $G$ with $\abs{K}=\omega(G)$.
    Let $A=N(K)$ and $B=V(G)\setminus (K\cup N(K))$.

    \begin{claim}
    $\omega(G[B])\le 3t-6$.
    \end{claim}
    \begin{subproof}
    Suppose that $G[B]$ has a clique $L$ of size $3t-5$. Let $P$ be a shortest path $v_0 \dd v_1 \dd \cdots \dd v_\ell$ from $K$ to $L$ where $v_0\in K$ and $v_\ell\in L$.
    By definition, $\ell \geq 2$.
    
    If $v_{\ell-1}$ has at least $t-1$ non-neighbors in $L$, then the graph induced by
    $(L \setminus N(v_{\ell -1}))\cup \{v_{\ell},v_{\ell-1},v_{\ell-2} \}$ contains a $t$-lollipop, a contradiction. 
    Therefore, $v_{\ell-1}$ has at least $2t-3$ neighbors in $L$.

    If $v_1$ has at least $t-1$ non-neighbors in $K$, then the graph induced by 
    $(K\setminus N(v_1))\cup\{v_0,v_1,v_2\}$ contains a $t$-lollipop, a contradiction. 
    Therefore $v_1$ has at least $2t-3$ neighbors in~$L$.
    So, $\ell > 2$ for otherwise, the graph on $N(v_1)$ contains an induced subgraph isomorphic to $2K_{2t-3}$. 
    
    As $t\geq 3$, we have $2t-3> t-1$.
    Then $t-1$ neighbors of $v_{\ell-1}$ in $L$ with $v_{\ell-1},v_{\ell-2},v_{\ell-3}$ 
    induce a $t$-lollipop, a contradiction.
    \end{subproof}

    For each subset $M$ of $K$ with $\abs{M}<2t-3$, 
    let $A_M$ denote the set of all vertices in $A$ that are anti-complete to $M$ and complete to $K\setminus M$. 
    Then,     
    $\omega(G[A_M])\le \abs{M}$, implying that 
    $\chi (G[A_M])\le m$.

    For each subset $N$ of $K$ with $\abs{N}=2t-3$ and each vertex $v\in K\setminus N$, 
    let $A'_{N,v}$ be the set of all vertices in $A$ 
    that are anti-complete to $N$
    and are adjacent to~$v$. 
    Since $G[N(v)]$ is $2K_{2t-3}$-free, 
    $\omega(G[A'_{N,v}])\le 2t-4$.
    This implies that $\chi(G[A'_{N,v}])\le m$.
    Observe that every vertex of $A$ is in $A_M$ or $A'_{N,v}$ for some choice of $M$, $N$, $v$.
    By the definition and the claim, $\chi(G) \leq \omega(G) + \chi(A) + \chi(B)\le \omega(G)+\chi(A)+m$, so we obtain
    \begin{equation}\label{eq:lollipop-bound}
    \chi(G)\le 
    \omega(G)+
    m\sum_{i=1}^{2t-4} \binom{\omega(G)}{i} 
    +m\binom{\omega(G)}{2t-3}(\omega(G)-(2t-3)) 
    +m,
    \end{equation}
    which is a polynomial in $\omega(G)$.
\end{proof}

We are now ready to show that the class of $t$-lollipop-free graphs is Pollyanna.

\begin{theorem}\label{lollipop-free2}
    For every integer $t \geq 1$,
    the class of $t$-lollipop-free graphs is $(3t-6)$-strongly Pollyanna.
\end{theorem}

\begin{proof}
    By \cref{cographs}, we may assume $t\geq3$.
    Let $\mathcal C$ be the class of $t$-lollipop-free graphs.
    Let $\mathcal C'$ be the class of ($t$-lollipop, $2K_{2t-3}^*$)-free graphs.
    Let $\mathcal F$ be a 
    hereditary class of graphs 
    and let $m$ be a positive integer
    such that $\chi(G)\le m$ whenever $G\in \mathcal F$ and $\omega(G)\le 3t-6$.  
    By~\cref{2k2t,fair-split-hom}, every graph in $\C \cap \F$ is either $2K_{2t-3}^*$-free 
    or has a homogeneous set.
Therefore, every graph in $\C\cap \F$ belongs to the closure of $\C'\cap \F$ under substitutions and disjoint unions.
    By~\cref{lem:lollipop-sub}, $\C'\cap \F$ is polynomially $\chi$-bounded and therefore 
    \cref{thm:substitution-orig} implies that $\C\cap \F$ is polynomially $\chi$-bounded.
\end{proof}

\section{Bowtie-free graphs}
\label{sec:bowtie}

A \emph{bowtie} is the graph on five vertices obtained from two copies of $K_2$ by adding a new vertex~$v$ and making it adjacent to all other vertices, see~\cref{fig:bowtie}.
In this section, we will show that bowtie-free graphs are $3$-strongly Pollyanna.

\begin{theorem}\label{bowtie:main}
    The class of bowtie-free graphs is $3$-strongly Pollyanna.
\end{theorem}

We do this by proving the following strengthening of \cref{bowtie:main}.

\begin{proposition}\label{prop:bowtie}
    Every bowtie-free graph $G$ 
    admits a partition of its vertex set into at most $f(\omega(G)) = \lceil\frac{1}{2}(\omega(G) + 3 \binom{\omega(G)}{3})\rceil +1 = \mathcal{O}(\omega(G)^3)$  sets 
    such that 
    one of the sets induces a $K_4$-free graph 
    and all other sets induce triangle-free graphs.
\end{proposition}

One of the key observations for the proof is that if $G$ is bowtie-free and has an edge $e$ not in any triangle, 
then $G\setminus e$ is also bowtie-free.
We will show that if $G$ is a counterexample to \cref{prop:bowtie} minimizing $|E(G)|$, then every edge of $G$ is in a triangle.
The following two lemmas show that some induced subgraphs are forbidden in such graphs.

\begin{lemma}\label{lem:bowtie-triangles-at-distance-at-least-2}
    If a graph $G$ has two disjoint cliques $A$ and~$B$ of size $4$ and $3$ respectively with exactly one edge between $A$ and~$B$, 
    then $G$ either has a bowtie as an induced subgraph 
    or has an edge that is not contained in a triangle.
\end{lemma}
\begin{proof}
    Suppose that every edge is contained in a triangle and that $G$ is bowtie-free.
    Let $a_1$, $a_2$, $a_3$, $a_4$ be the vertices of $A$ 
    and $b_1$, $b_2$, $b_3$ be the vertices of $B$. We may assume that $e=a_1b_1$ is the unique edge between $A$ and~$B$.
    Since $e$ is contained in a triangle, there is a vertex~$x\notin A\cup B$ adjacent to both $a_1$ and $b_1$. 
    As $\{a_1,x,b_1,b_2,b_3\}$ does not induce a bowtie, 
    we may assume that $x$ is adjacent to $b_2$.
    Similarly, as $\{b_1,x,a_1,a_i,a_j\}$ does not induce a bowtie for all $2\le i<j\le 4$, 
    we may assume that $x$ is adjacent to $a_2$ and $a_3$.
    Then $\{x,a_2,a_3,b_1,b_2\}$ induces a bowtie, a contradiction.
\end{proof}

\begin{lemma}\label{bowtie-triangles-at-distance-at-least-3}
    If a graph $G$ has two disjoint and anti-complete cliques $A$ and~$B$ of size $4$ and $3$ respectively 
    and a vertex~$v$ with at least one neighbor in each of $A$ and~$B$, 
    then 
    $G$ either has a bowtie as an induced subgraph 
    or has an edge that is not contained in a triangle.
\end{lemma}
\begin{proof}
    Suppose that $G$ is bowtie-free and that every edge is contained in a triangle and suppose there is some $v \in V(G)$ with at least one neighbor in each of $A$ and~$B$.
\begin{claim}\label{bowtie:A-and-B}
     For every $u \in V(G)$ with at least one neighbor in each of $A$ and~$B$, $u$ has at most one neighbor in $B$.
    \end{claim}
    \begin{subproof}
        If $u$ has at least two neighbors in $B$, 
        then 
        $u$ has exactly one neighbor in $A$ because $G$ is bowtie-free.
        It follows that $A$ and $\{u\}\cup (N(u)\cap B)$ are two cliques of size $4$ and $3$ respectively with exactly one edge between $A$ and $\{u\}\cup (N(u)\cap B)$, contradicting~\cref{lem:bowtie-triangles-at-distance-at-least-2}.
    \end{subproof}
Hence, we may assume $v$ has exactly one neighbor $b \in B$.
\begin{claim}
        $|N(v) \cap A| \geq 2$.
    \end{claim}
    \begin{subproof}
    Suppose that $v$ has exactly one neighbor~$a_1$ in $A$.
    As there is a triangle containing $a_1v$, there is a common neighbor~$x\notin A\cup B$ of $a_1$ and $v$. 
    Since $G[A\cup \{x,v\}]$ is bowtie-free, $x$ is adjacent to at least three vertices $a_1,a_2,a_3$ in $A$.
    Since $G[\{a_2,a_3,x\} \cup B]$ is bowtie-free, it follows that $x$ has at most one neighbor in $B$.
    By~\cref{lem:bowtie-triangles-at-distance-at-least-2},
    $x$ is adjacent to no vertex in $B$.

    There is a common neighbor~$y\notin A\cup B$ of $v$ and $b$
    and $y$ is adjacent to at least two vertices in $B$.
    Hence $y$ cannot be adjacent to two vertices of $A$ for otherwise $G[\{y\} \cup N(y)]$ would contain a bowtie.
    By \cref{lem:bowtie-triangles-at-distance-at-least-2}, $y$ has no neighbor in $A$.
    Note that $y\not=x$ since $x$ is not adjacent to $b_1$.
    
    Since $G[\{v,a_1,x,y,b\}]$ is not a bowtie,
    $x$ is adjacent to $y$.
    Then $G$ has two cliques $\{x\} \cup (N(x)\cap A)$ and $\{y\}\cup (N(y)\cap B)$ of cardinality at least $4$ and $3$ respectively with exactly one edge $xy$ between $\{x\} \cup (N(x)\cap A)$ and $\{y\}\cup (N(y)\cap B)$, contradicting~\cref{lem:bowtie-triangles-at-distance-at-least-2}.
    \end{subproof}

    Now it remains to consider the case where $v$ has at least two neighbors in $A$.    
Let $y$ be a common neighbor of $v$ and $b$. 
    Since $\{v,y\}\cup B$ does not induce a bowtie, 
    $y$ has at least two neighbors in $B$.
    Then by \cref{bowtie:A-and-B}, $y$ has no neighbor in $A$.
    But then the graph induced by $A \cup \{v, y, b\}$ contains a bowtie, a contradiction.
    This completes the proof.
\end{proof}

We are now ready to prove \cref{prop:bowtie} (and thus \cref{bowtie:main}).
\begin{proof}[Proof of \cref{prop:bowtie}]
    We proceed by induction on $\abs{E(G)}$.
    We may assume that $G$ is connected.
    The statement is trivial if $\omega(G)<4$ and so we may assume that $\omega(G)\ge 4$.

    Suppose there is some $e \in E(G)$ such that $e$ is not contained in any triangle.
    Let $G' = G \setminus e$.
    Then, $G'$ is bowtie-free and $\omega(G')=\omega(G)$.
    By the inductive hypothesis, $V(G)$ admits a partition into sets $X_1$, $X_2$, $\ldots$, $X_k$ such that $k\le f(\omega(G))$,~
    $\omega(G'[X_1])\le 3$, 
    and $G'[X_i]$ is triangle-free for all $i\in\{2,3,\ldots,k\}$.
    Since $e$ is not in any triangle of $G$, we deduce that $\omega(G[X_1])\le 3$ and $G[X_i]$ is triangle-free for all $i\in \{2,3,\ldots,k\}$. 
    Therefore, we may assume that every edge is in a triangle. 

    Let $K$ be a maximum clique in~$G$.
    Then $|K|=\omega(G) \ge 4$.    
    Suppose that there is a vertex $v$ such that the distance from $v$ to $K$ is $3$.
    Let $v_0 \dd v_1 \dd v_2 \dd v_3$ be a shortest path  from $K$ to~$v$ where $v_0\in K$ and $v_3=v$.
    There is a common neighbor~$x$ of $v_2$ and $v_3$.
    Since the distance between~$K$ and $v_3$ is equal to $3$, the two cliques $K$ and $\{v_2,v_3,x\}$ are disjoint and anti-complete.
    Then $v_1$ has neighbors in both~$K$ and $\{v_2,v_3,x\}$, contradicting~\cref{bowtie-triangles-at-distance-at-least-3}. 
    Therefore, every vertex of~$G$ is within distance $2$ from~$K$.

    Let $A$ be the set of vertices of distance $1$ from $K$ 
    and $B=V(G)\setminus (K\cup A)$.
    Note that every vertex in $B$ has a neighbor in $A$
    and every vertex in $A$ has at least one non-neighbor in~$K$.
    By~\cref{bowtie-triangles-at-distance-at-least-3}, 
    $G[B]$ is triangle-free.
    For each vertex~$x\in K$, let 
    $S_x$ be the set of vertices in~$A$ complete to $K\setminus\{x\}$.
    Since $K$ is a maximum clique, $S_x\cup\{x\}$ is independent.
    For distinct vertices $x,y,z\in K$, 
    let $T_{x,y,z}=(A\cap N_G(z))\setminus (N_G(x)\cup N_G(y))$.
    Since $G$ is bowtie-free, $T_{x,y,z}$ is independent.
    
    By definition, every $a \in A$ with at least two non-neighbors in $K$ is in $T_{x,y,z}$ for some choice of $x,y, z \in K$ and every $a \in A$ with exactly one non-neighbor $x \in K$ is in $S_x$.
Therefore, we have a partition of $V(G)$
    into $S_x\cup\{x\}$ for $x\in K$, 
    $T_{x,y,z}$ for $x,y,z\in K$,
    and $B$.
Note that every set except $B$ in our partition is stable, so we can merge any other two sets in our partition to obtain another triangle-free set.
    So we obtain a partition of $V(G)$ into 
    at most $\lceil\frac{1}{2}(\omega(G) + 3 \binom{\omega(G)}{3})\rceil +1$ sets.  
\end{proof}

\section{Bull-free graphs}
\label{sec:bull}
In this section, we will show that the class of bull-free graphs is Pollyanna.
We will begin by reducing the problem of showing the class of bull-free graphs is Pollyanna to showing that a simpler subclass of bull-free graphs is Pollyanna using structural results about bull-free graphs by Chudnovsky and Safra~\cite{ErdosHajnalBulls}.
We begin with some definitions.
For a subgraph~$H$ of a graph $G$, we say $v \in V(G) \setminus V(H)$ is a \emph{center} for $H$ if it is complete to $V(H)$.
If $v$ is a center for $H$ in $\overline{G}$, we say $v$ is an \emph{anticenter} for $H$ in $G$.
    We say a bull-free graph $G$ is \emph{basic} if neither $G$ nor $\overline{G}$ contains an odd hole with both a center and an anticenter.
We say a graph $G$ is \emph{locally perfect} if for every $v \in V(G)$, the graph induced by $N_G(v)$ is perfect.

We will show that if the class of locally perfect basic bull-free graphs is Pollyanna, then so is the class of bull-free graphs.
We will require the following theorem by Chudnovsky and Safra~\cite{ErdosHajnalBulls}, which also appears in a paper of Chudnovsky~\cite{Chudnovsky2012a} in greater generality according to \cite{ErdosHajnalBulls}.
\begin{theorem}[Chudnovsky and Safra~{\cite[1.4]{ErdosHajnalBulls}}]\label{every-composite-has-a-hom-set}
    Every bull-free graph can be obtained via substitution from basic bull-free graphs. 
\end{theorem}

\begin{theorem}[Chudnovsky and Safra~{\cite[4.3]{ErdosHajnalBulls}}]\label{thm:basic-nbr-nonnbr-perfect}
    If $G$ is a basic bull-free graph,
    then $G[N(v)]$ or $G \setminus (N(v)\cup \{v\})$ is perfect 
    for every vertex $v$ of~$G$.
\end{theorem}

\begin{corr}\label{thm:ets-locally-perfect-basic}
    Let $\mathcal{F}$ be a hereditary class of graphs.
    If the class of locally perfect basic bull-free graphs in $\F$ is polynomially $\chi$-bounded,
    then so is the class of bull-free graphs in $\F$.
\end{corr}
\begin{proof}
    Let $\C$ denote the class of basic bull-free graphs in $\F$.
    Note that $\C$ is hereditary.
    By \cref{every-composite-has-a-hom-set,thm:substitution-orig}, it is enough to show that $\C$ is polynomially $\chi$-bounded.

    Suppose that there is a polynomial $f$ such that every locally perfect basic bull-free graph~$G$ in~$\F$ satisfies $\chi(G) \leq f(\omega(G))$.
    We may assume that $f(n)\ge n$ for all positive integers~$n$.

    We claim that every $G \in \C$ satisfies $\chi(G) \leq \sum_{k=1}^{\omega(G)}f(k)$. 
    We proceed by the induction on $\omega(G)$.
    The statement is trivial if $\omega(G)\le 1$ and so we assume that $\omega(G)>1$.
    We may assume that $G$ is not locally perfect because otherwise $\chi(G)\le f(\omega(G))$.
    So there is a vertex~$v$ such that $G[N(v)]$ is not perfect.
    By \cref{thm:basic-nbr-nonnbr-perfect}, $G \setminus (N(v)\cup\{v\})$ is perfect and so is $G\setminus N(v) $. Therefore, $\chi(G \setminus N(v))\le \omega(G)\le f(\omega(G))$.
    Since $\omega(G[N(v)])<\omega(G)$,
    by the induction hypothesis, $\chi(G[N(v)])\le \sum_{k=1}^{\omega(G)-1} f(k)$.
    This completes the proof because $\chi(G) \leq \chi(G[N(v)])+\chi(G\setminus N(v))$.
\end{proof}

Hence, we only need to show that the class of locally perfect bull-free graphs is Pollyanna.
We will do so by invoking results by Chudnovsky~\cite{Chudnovsky2012a} about \say{elementary} and \say{non-elementary} bull-free graphs.
A bull-free graph is \emph{elementary} if it does not contain a path of length three with both a center and an anticenter.
For a positive integer $k$, 
we say a graph $G$ is \emph{$k$-perfect} if $V(G)$ can be partitioned into at most $k$ sets each of which induces a perfect graph. 
We will first prove the following proposition on elementary locally perfect bull-free graphs.

\begin{restatable}{proposition}{elementarykperfect}\label{elementary-bull-free-constant-perfect}
For every $2$-good class $\F$ of graphs, 
    there is a positive integer~$\gamma$ such that 
    every elementary locally perfect bull-free graph in $\F$ is $\gamma$-perfect.
\end{restatable}

We then use \cref{elementary-bull-free-constant-perfect} to prove the following for locally perfect bull-free graphs. Its proof uses trigraphs, which we will introduce in the next subsection.

\begin{restatable}{proposition}{nonelementary}\label{thm:non-elementary-locally-perfect}
For every $2$-good class $\F$ of graphs, 
    there is a positive integer~$c_\F$ such that 
    every locally perfect bull-free graph is $c_\F$-perfect.
\end{restatable}

It is now straightforward to prove that the class of bull-free graphs is Pollyanna if we assume~\cref{thm:non-elementary-locally-perfect}.
As we remarked in the introduction, we will actually prove that the class of bull-free graphs is $2$-strongly Pollyanna which is a stronger statement.

\begin{theorem}\label{bulls-main}
    The class of bull-free graphs is $2$-strongly Pollyanna.
\end{theorem}
\begin{proof}[Proof assuming~\cref{thm:non-elementary-locally-perfect}]
    By \cref{thm:non-elementary-locally-perfect}, the class of locally perfect bull-free graphs is $2$-strongly Pollyanna.
    Hence, we obtain that the class of bull-free graphs is $2$-strongly Pollyanna by applying \cref{thm:ets-locally-perfect-basic}.
\end{proof}

\subsection{Trigraphs}
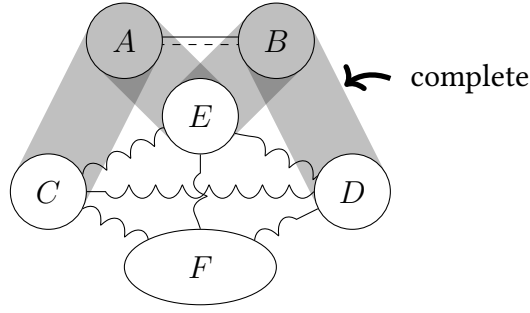
\begin{figure}
    \centering
    \begin{tikzpicture}
        \node [circle,draw,fill=lightgray,minimum size=1cm] (a) at (-1,1) {$A$};
        \node [circle,draw,fill=lightgray,minimum size=1cm] (b) at (1,1) {$B$};
        \node [circle,draw,minimum size=1cm,fill=white] (c) at (-2,-1) {$C$};
        \node [circle,draw,minimum size=1cm,fill=white] (d) at (2,-1) {$D$};
        \node [circle,draw,minimum size=1cm,fill=white] (e) at (0,0) {$E$};
        \node [ellipse,draw,minimum height=1cm,
        minimum width=2cm,fill=white] (f) at (0,-2) {$F$};
        \begin{pgfonlayer}{background}
        \draw [line width=1cm,opacity=0.25]
        (2,-1)--(1,1);
        \draw [line width=1cm,opacity=0.25](1,1)--(0,0);
        \draw [line width=1cm,opacity=0.25](0,0)--(-1,1);
        \draw [line width=1cm,opacity=0.25,](-1,1)--(-2,-1);
        \draw [decorate,decoration=coil] (c)--(e)--(d)--(c);
        \draw (-1,1.05)--(1,1.05);
        \draw [dashed] (-1,0.95)--(1,0.95);
        \draw [decorate,decoration=coil](c)--(f)--(e);
        \draw [decorate,decoration=coil](f)--(d);
        \end{pgfonlayer}
\draw [->,bend right,line width=2pt] (2.5,0.5) node [inner sep=0pt, minimum width=3pt,label=right:{complete}]{} to (1.9,0.4);
    \end{tikzpicture}
\caption{A homogeneous pair.}
    \label{fig:c5leafstar}
\end{figure}

To describe the necessary results from a paper of Chudnovsky~\cite{Chudnovsky2012a}, we will need to use a generalization of graphs called \emph{trigraphs}.
For a set $X$, let us write $\binom{X}{2}$ to denote all $2$-element subsets of $X$. 
A \emph{trigraph}~$G$ is an object consisting of a finite set $V(G)$, called the \emph{vertex set} of~$G$, and the \emph{adjacency function} $\theta:\binom{V(G)}{2}\to\{-1,0,1\}$. 
Two distinct vertices $u$ and $v$ of $G$ are \emph{strongly adjacent} if $\theta(\{u,v\})=1$
\emph{strongly anti-adjacent} if $\theta(\{u,v\})=-1$,
and \emph{semi-adjacent} if $\theta(\{u,v\})=0$.
If $u$ and $v$ are semi-adjacent, we say the pair $\{u,v\}$ is a \emph{switchable} pair.
We regard graphs as trigraphs without semi-adjacent pairs of vertices.

Two vertices of a trigraph are \emph{adjacent}
if they are strongly adjacent or semi-adjacent.
Similarly, two vertices of a trigraph are \emph{anti-adjacent}
if they are strongly anti-adjacent or semi-adjacent.
For two disjoint subsets $A$ and $B$ of vertices of a trigraph, 
$A$ is \emph{strongly complete} to~$B$ 
if every vertex in $A$ is strongly adjacent to every vertex in~$B$,
and 
\emph{strongly anti-complete} 
if every vertex in $A$ is strongly anti-adjacent to every vertex in~$B$.
If a vertex $x$ is adjacent to a vertex $y$, then $y$ is called a \emph{neighbor} of $x$. We write $N_G(x)$ to denote the set of all neighbors of~$x$. We sometimes omit the subscript if it is clear from the context.

The complement $\overline{G}$ of a trigraph $G=(V,\theta)$ is a trigraph on the same vertex set $V(G)$ with the adjacency function $\overline{\theta}=-\theta$. 
For a set $X$ of vertices, we write $G[X]$ to denote the subtrigraph induced by $X$, which has the vertex set $X$ and the adjacency function is the restriction of $\theta$ to $\binom{X}{2}$.
We say that $H$ is an induced subtrigraph of $G$ if $H=G[X]$ for some $X\subseteq V(G)$.
We write $G\setminus X$ to denote the trigraph $G[V(G)\setminus X]$.
Isomorphisms between trigraphs are defined as usual.

A set $X$ of vertices of a trigraph is a \emph{strong clique} if $x$ and $y$ are strongly adjacent for all distinct $x,y\in X$.

For a trigraph $G$, let $\hat G$ be a graph on $V(G)$ such that two vertices of $\hat G$ are adjacent if and only if they are adjacent in $G$. 
We call $\hat G$ the \emph{full realization} of $G$.
We say that $G$ is \emph{connected} if $\hat G$ is connected. A \emph{connected component} of a trigraph is a maximal connected induced subtrigraph.

A graph is a \emph{realization} of a trigraph~$G$ 
if its vertex set is equal to $V(G)$
and its edge set is the set of all strongly adjacent pairs and possibly some switchable pairs of $G$. 
A trigraph~$G$ \emph{contains} a graph $H$
if $G$ has a realization containing an induced subgraph isomorphic to~$H$.

A \emph{homogeneous set} of a trigraph~$G$ is a proper subset $X$ of $V(G)$ with at least two vertices such that every vertex in $V(G)\setminus X$ is either strongly complete or strongly anti-complete to~$X$.

For a trigraph $G$, a pair $(A, B)$ of disjoint nonempty subsets of $V(G)$ is a \emph{homogeneous pair} if 
$V(G) \setminus (A \cup B)$ can be partitioned into four (possibly empty) sets $C$, $D$, $E$, and $F$  such that 
\begin{itemize}
    \item $C$ is strongly complete to $A$ and strongly anti-complete to $B$,
    \item $D$ is strongly complete to $B$ and strongly anti-complete to $A$, 
    \item $E$ is strongly complete to both $A$ and $B$, and
    \item $F$ is strongly anti-complete to both $A$ and $B$.
\end{itemize} 
We say the pair $(A, B)$ is \emph{tame} if 
\begin{itemize}
    \item $|V(G)| -2 > |A| + |B|>2$ and 
    \item $A$ is not strongly complete to $B$ and not strongly anti-complete to $B$.
\end{itemize}
A trigraph $G$ admits a \emph{homogeneous pair decomposition} if it has a tame homogeneous pair.
We say that a homogeneous pair $(A,B)$ is \emph{proper} if it is tame and 
both $C$ and $D$ are nonempty.
We say that a homogeneous pair $(A,B)$ is \emph{small} if it is tame and $\abs{A\cup B}\le 6$.
See \cref{fig:c5leafstar} for an illustration of a homogeneous pair.

We say a tame homogeneous pair $(A,B)$ of a trigraph~$G$ 
is \emph{dominated} if there exist (possibly identical) vertices $v$ and $w$ in $V(G)\setminus (A\cup B)$ such that $v$ is strongly complete to $A$ and $w$ is strongly complete to $B$. In other words, $E\neq \emptyset$ or both $C$ and $D$ are nonempty.

For two homogeneous pairs $(A_1, B_1)$ and $(A_2, B_2)$  of a trigraph, 
we say $(A_2,B_2)$ contains $(A_1,B_1)$, denoted by $(A_2, B_2)\subseteq (A_1, B_1)$, if $A_1 \subseteq A_2$ and $B_1\subseteq B_2$.  
In addition, we say $(A_2,B_2)$ contains $(A_1,B_1)$ \emph{properly} if 
$(A_2, B_2)\subseteq (A_1, B_1)$ and $(A_2, B_2)\neq (A_1, B_1)$.
A tame homogeneous pair of a trigraph is \emph{maximal} if it is not properly contained by any tame homogeneous pair.

We say a trigraph is \emph{monogamous} if every vertex belongs to at most one switchable pair.
\emph{Shrinking} a tame homogeneous pair $(A,B)$ in a trigraph 
is an operation to shrink $A$ into a single vertex~$a$, shrink $B$ into a single vertex $b$, 
and make the pair $\{a, b\}$ a switchable pair.

\subsection{The elementary locally-perfect case} \label{proof-of-elementary-bull-free-kperfect}
In this subsection, we will prove \cref{elementary-bull-free-constant-perfect}.
The class $\mathcal T_1$ of trigraphs is defined in Chudnovsky~\cite{Chudnovsky2012}.
Thomass\'e, Trotignon, and Vu\v{s}kovi\'c~\cite[Subsection 2.2]{TTV2017} observed the following.
\begin{obs}
    Every graph~$G$ in $\mathcal T_1$ has a partition $(X,K_1,K_2,\ldots,K_t)$ of its vertex set into  sets for some $t\ge0$ such that 
    $G[X]$ does not contain a triangle 
    and 
    $K_1,\ldots,K_t$ are cliques that are pairwise anti-complete.
\end{obs}
Hence, we immediately deduce the following.
\begin{obs}\label{tau1-is-k-perfect}
    Every graph $G$ in $\mathcal T_1$ admits a partition of its vertex set into two sets $(X,Y)$ such that 
    $G[X]$ is triangle-free 
    and $G[Y]$ is perfect.
\end{obs}

\begin{lemma}\label{lem:2-path}
    If $G$ is a graph with no homogeneous set and $X$ is a proper subset of $G$ that is not stable, then there is an induced path $x_1\dd x_2\dd y$ such that $x_1,x_2\in X$ and $y\in V(G)\setminus X$.
\end{lemma}
\begin{proof}
    Suppose not. Since $X$ is not stable, $G[X]$ contains a component $C$ with at least two vertices.
    Since $V(C)$ is not homogeneous, there is $y\in  G\setminus V(C)$ such that $y$ is neither complete nor anti-complete to $V(C)$. 
    Clearly $y\notin X$ and since $C$ is connected, there exist an edge $x_1x_2$ of $C$ such that 
    $y$ is adjacent to $x_2$ and non-adjacent to $x_1$.
\end{proof}

A \emph{gem} is the $5$-vertex graph obtained from the path of length $3$ by adding a vertex adjacent to all other vertices.
Note that every gem-free bull-free graph is elementary.
We first aim to show \cref{elementary-bull-free-constant-perfect} restricted to gem-free graphs.

Here is an easy lemma based on \cref{bullperfect}.
\begin{lemma}\label{lem:no-odd-hole}
    Let $G$ be a bull-free gem-free graph.
    Then $G$ is perfect if and only if $G$ has no odd hole. \qed
\end{lemma}

\begin{lemma}\label{lem:hompair}
    Let $G$ be a bull-free gem-free graph.
    Let $(A,B)$ be a tame homogeneous pair of $G$ and let $C$, $D$, $E$, $F$ be as in the definition of a homogeneous pair. If $G$ has no homogeneous set, then the following hold.
    \begin{enumerate}[label=\rm(\roman*)]
        \item\label{item:stable} If $A$ is not stable, then $C$ is anti-complete to $F$ and complete to $E$.
        \item\label{item:stable2} If $B$ is not stable, then $D$ is anti-complete to $F$ and complete to $E$.
        \item\label{item:nonedges}If $A$ is not a clique, then $E$ is anti-complete to $C$ and complete to $D$.
        \item\label{item:nonedges2} If $B$ is not a clique, then $E$ is anti-complete to $D$ and complete to $C$.
        \item\label{item:completetoE} $E$ is complete to $C$ or $D$.
\end{enumerate}
\end{lemma}
We remark that 7.4 of~\cite{Chudnovsky2012} implies half of each of \ref{item:stable}--\ref{item:nonedges2}.
\begin{proof}
    Suppose $A$ is not stable. By \cref{lem:2-path} and the definition of homogeneous pairs, there exist $a_1, a_2 \in A$ and $b \in B$ such that $b \dd a_1 \dd a_2$ is an induced path of $G$.
    Then, if there is some $c \in C$ adjacent to some $f \in F$, the graph on $\{f, c, a_1, b, a_2 \}$ induces a bull, a contradiction.
    If there is some $c \in C$ non-adjacent to some vertex $x \in E$,
    then $c \dd a_2 \dd x \dd b$ is an induced path of length $3$ with a center~$a_1$, a contradiction. See \cref{fig:nonedges}.
    This proves \ref{item:stable}.
    By symmetry, we also have \ref{item:stable2}.
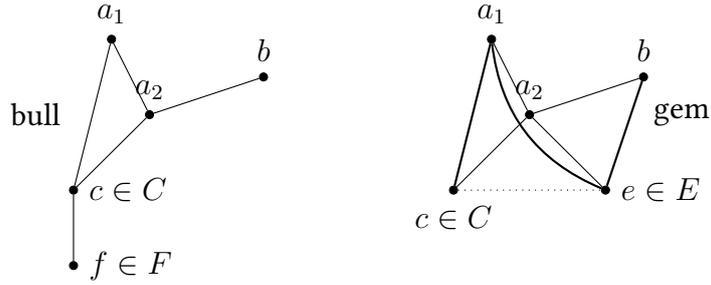
\begin{figure}
        \centering

        \begin{tikzpicture}
            \tikzstyle{v}=[circle, draw, solid, fill=black, inner sep=0pt, minimum width=3pt]
            \draw (-.5,1) node [v,label=$a_1$](a1) {};
            \draw (0,0) node[v,label=above:$a_2$] (a2) {};
            \draw (1.5,0.5) node[v,label=$b$] (b){};
            \draw (a2)--(b);
\draw (-1,-1) node[v,label=right:$c\in C$] (c){};
            \draw (-1,-2) node[v,label=right:$f\in F$] (f){};
            \draw (c)--(f);
            \draw (a1)--(a2)--(c)--(a1);
            \draw (-1.5,0) node {bull};

            \begin{scope}[xshift=5cm]
            \tikzstyle{v}=[circle, draw, solid, fill=black, inner sep=0pt, minimum width=3pt]
            \draw (-.5,1) node [v,label=$a_1$](a1) {};
            \draw (0,0) node[v,label=above:$a_2$] (a2) {};
            \draw (1.5,0.5) node[v,label=$b$] (b){};
            \draw (a2)--(b);
\draw (-1,-1) node[v,label=below:$c\in C$] (c){};
            \draw (1,-1) node[v,label=right:$e\in E$] (e){};
            \draw [dotted](c)--(e);
            \draw [bend right,thick](a1) to (e);
            \draw (a2)--(e);
            \draw [thick](b)--(e);
            \draw (a1)--(a2)--(c);
            \draw [thick] (c)--(a1);
            \draw (2,0) node {gem};
            \end{scope}
        \end{tikzpicture}
\caption{An illustration of \cref{lem:hompair}\ref{item:stable}.}
        \label{fig:nonedges}
    \end{figure}
    
    Let us now prove \ref{item:nonedges}.
    Suppose $A$ is not a clique.
    By applying \cref{lem:2-path} to $\overline G$, 
    we deduce that there exist $a_1, a_2 \in A$ and $b \in B$ such that $b \dd a_1 \dd a_2$ is an induced path of $\overline{G}$.
    If there is a vertex $x\in E$ adjacent to a vertex $c\in C$, 
    then $a_1\dd c\dd a_2\dd b$ is an induced path with a center $x$, a contradiction.
If some vertex $x \in E$ is non-adjacent to some $d \in D$,
    then $\{a_1, b, x, a_2, d\}$ induces a bull.
    See \cref{fig:nonedges}.
    This proves \ref{item:nonedges}.
    By symmetry between $A$ and $B$, we deduce \ref{item:nonedges2}.

Since $(A,B)$ is tame, $|A| > 1$ or $|B| > 1$.
    Thus, it follows from
    \ref{item:stable}, \ref{item:stable2}, \ref{item:nonedges}, and \ref{item:nonedges2} that $E$ is complete to $C$ or $D$, proving \ref{item:completetoE}.
\end{proof}

Based on papers of Chudnovsky~\cite{Chudnovsky2012a,Chudnovsky2012}, bull-free graphs admit the following decomposition, summarized by Thomass\'e, Trotignon, and Vu\v skovi\'c~\cite{TTV2017}.
We state it for graphs instead of trigraphs.

\begin{theorem}[Chudnovsky~\cite{Chudnovsky2012a,Chudnovsky2012}; see Thomass\'e, Trotignon, and Vu\v skovi\'c~{\cite[Theorem 2.1]{TTV2017}}]
    \label{thm:bullfree}
    Every bull-free graph $G$ satisfies one of the following.
    \begin{enumerate}[label=\rm(\roman*)]
        \item $\abs{V(G)}\le 8$.
        \item $G$ or $\overline G$ belongs to $\mathcal T_1$. 
        \item $G$ has a homogeneous set.
        \item $G$ has a proper homogeneous pair. 
        \item $G$ has a small homogeneous pair.
    \end{enumerate}
\end{theorem}
\begin{proposition}\label{prop:gemfree-main}
    For every $2$-good class $\F$ of graphs, there is a positive integer $\gamma$
    such that
    every bull-free gem-free graph in $\F$ is $\gamma$-perfect.
\end{proposition}
\begin{proof}
    By definition of $2$-good, $\F$ is hereditary and 
    there exists a positive integer~$\tau$ such that every triangle-free graph in $\F$ is $\tau$-colorable.
    Let $\gamma = \max\{6, \tau+1\}$.
    Let $G$ be a bull-free gem-free graph in $\F$. 
    
    Suppose that $G$ is not $\gamma$-perfect.
    We choose such a $G$ with the minimum $\abs{V(G)}$.
    Since the disjoint union of perfect graphs is perfect,
    $G$ is connected.
    Since $G$ is gem-free and since $P_4$-free graphs are perfect, 
    for every vertex $v$ of $G$, 
    $G[N_G(v)\cup\{v\}]$ is perfect and therefore 
    \stmt{
        $G$ has no dominating set of at most $\gamma$ vertices
\label{claim:domsets}}
    and $G$ is locally perfect.

    \begin{claim}
        $G$ does not admit a homogeneous set.
    \end{claim}
    \begin{subproof}
Suppose $S \subset V(G)$ is a homogeneous set in $G$.
    Since $G$ is connected, there is some $v \in V(G) \setminus S$ such that $v$ is complete to $S$.
    Hence, $G[S]$ is perfect because $G$ is locally perfect.
    Let $w\in S$ and $G'=G\setminus (S\setminus\{w\})$.
Since $G'$ is an induced subgraph of $G$, $G'$ is also bull-free and gem-free
    and therefore by the minimality of $G$, it follows that $G'$ is $\gamma$-perfect.
    Let $(V_1, V_2, \dots, V_{\gamma})$ be a partition of $V(G')$ such that $G[V_i]$ is perfect for each $i \in \{1,2, \dots, \gamma\}$.
    Without loss of generality, $w \in V_1$.
    Then, since perfect graphs are closed under substitution by \cref{lovaszreplacement} and $G[S]$ is perfect,
    $G[V_1\cup S]$ is perfect.
    Hence, $G$ is $\gamma$-perfect, a contradiction.
\end{subproof}

    By \cref{tau1-is-k-perfect}, every graph in $\mathcal T_1$ is $(\tau+1)$-perfect and so is every graph in $\overline{\mathcal T}_1$.    
    Thus, neither $G$ nor $\overline G$ is in $\mathcal T_1$. 
    Since every graph on at most $4$ vertices is perfect, every graph on at most $8$ vertices is $2$-perfect. Therefore, $\abs{V(G)}>8$.

    By \cref{thm:bullfree}, $G$ admits a proper or small homogeneous pair $(A, B)$.
    Let $C$, $D$, $E$, $F$ be as in the definition of a homogeneous pair.

    \begin{claim}\label{claim:F-nonempty}
        $F\neq\emptyset$.
    \end{claim}
    \begin{subproof} 
        Suppose that $F=\emptyset$. 
    If $C\cup D\neq\emptyset$ or $E\neq\emptyset$, then there is a dominating set of $G$ consisting of at most $4$ vertices made by choosing $1$ vertex from each of $A$ and $B$ and choosing $1$ vertex either from $E$ or from each of $C$ and $D$. Since $\gamma\ge 4$, this contradicts \cref{claim:domsets}. 
    Therefore, 
$E=\emptyset$ and $C$ or $D$ is empty.
By the symmetry between $A$ and $B$, we may assume $D = \emptyset$.
    Then, since $(A,B)$ is a tame homogeneous pair and $F \cup E \cup D = \emptyset$, it follows that $\abs{C}\ge 3$.
    But then $C$ is a homogeneous set, a contradiction. 
    Therefore, we deduce that $F\neq\emptyset$.
    \end{subproof}
    \begin{claim}\label{claim:proper}
        If $E=\emptyset$, then $(A,B)$ is proper. 
    \end{claim}
    \begin{subproof}
        By the assumption, $(A,B)$ is small.
        By symmetry, suppose that $D=E=\emptyset$. 
        By the induction hypothesis, there exists a partition $(V_1,V_2,\ldots,V_\gamma)$ of $A\cup C\cup F$ such that $G[V_i]$ is perfect for all $i\in \{1,2,\ldots,\gamma\}$. 
        We may assume that $A\cap V_i=\emptyset$ for all $i\le \abs{B}$
        because $\gamma\ge \abs{A\cup B}$.
        Let $w_1,w_2,\ldots,w_{\abs{B}}$ be the vertices in $B$.
        For $i\in \{1,2,\ldots,\abs{B}\}$, let $V_i':=V_i\cup \{w_i\}$.
        Since $w_i$ is isolated in $G[V_i']$, $G[V_i']$ is perfect. 
        For $i>\abs{B}$, define $V_i':=V_i$. 
        Then $G[V_i']$ is perfect for every $i\in \{1,2,\ldots,\gamma\}$ and $\bigcup_{i=1}^\gamma V_i'=V(G)$. Thus, $G$ is $\gamma$-perfect, a contradiction.
    \end{subproof}
    \begin{claim}\label{claim:ab-perfect}
        $G[A]$ and $G[B]$ are $P_4$-free, so perfect.
    \end{claim} 
    \begin{subproof}      
        It is trivial if $(A,B)$ is proper because $G$ is gem-free.
        By \cref{claim:proper}, we may assume that $E\neq\emptyset$. 
        This implies that $G[A\cup B]$ is $P_4$-free, because $G$ is gem-free.
    \end{subproof}

    \begin{claim}
        If $E = \emptyset$, then $A$ or $B$ is stable. \label{claim:A-or-B-stable}
    \end{claim}
    \begin{subproof}
          Suppose neither $A$ nor $B$ is stable.
    By \ref{item:stable} and \ref{item:stable2} of \cref{lem:hompair}, $C\cup D$ is anti-complete to $F$.
    However, by \cref{claim:F-nonempty}, $F \neq \emptyset$ and therefore $G$ is disconnected, a contradiction.  
    \end{subproof}

    By the definition of a tame homogeneous pair, there exist some $a \in A$ and $b \in B$ such that $ab$ is an edge of $G$.
    Let $G'$ denote the graph obtained from $G$ by deleting $(A \cup B) \setminus \{a, b\}$.
    By the definition of a tame homogeneous pair, $\abs{V(G')} < \abs{V(G)}$.
    By the choice of $G$, 
    there is a list $H_1, H_2, \dots, H_{\gamma}$ of perfect induced subgraphs of $G'$ that cover the vertex set of $G'$.
    Let $i,j \in \{1,2, \dots, \gamma\}$ be such that $a \in H_i$ and $b \in H_j$.
    If $i \neq j$, then $G[V(H_i) \cup A]$ and $G[V(H_j) \cup B]$ are obtained from $H_i$ and $H_j$ respectively via substitution.
    So by \cref{lovaszreplacement,claim:ab-perfect}, they are both perfect graphs.
    And therefore $G$ is $\gamma$-perfect, a contradiction.
    
    Hence, $i = j$.
    Let $H$ be the graph $G[V(H_i) \cup A \cup B]$.
    To get a contradiction, it is enough to show that $H$ is a perfect graph, because this would imply that $G$ is $\gamma$-perfect.
    Suppose that $H$ is not perfect. 
    Then by \cref{lem:no-odd-hole}, it contains an induced subgraph $X$ that is an odd hole.

    \Needspace{3\baselineskip}
    \begin{claim}
        $X$ contains vertices $a' \in A$ and $b' \in B$ where $a'$ and $b'$ are not adjacent.
        \label{claim:non-adjacent-a-b-prime}
    \end{claim}
    \begin{subproof}
    Since both $H\setminus A$ and $H\setminus B$ are perfect by \cref{lovaszreplacement}, $V(X) \cap A$ and $V(X) \cap B$ are both nonempty.
    Note that $G[(V(X) \setminus (A \cup B)) \cup \{a,b \}]$ is an induced subgraph of $H_i$ and therefore perfect.
    Moreover, $V(X) \cap A$ and $V(X) \cap B$ are not complete to each other, for otherwise $X$ can be obtained from $G[(V(X) \setminus (A \cup B)) \cup \{a,b \}]$ by substituting in $G[V(X) \cap A]$ for $a$ and $G[V(X) \cap B]$ for $b$, and therefore $X$ would be perfect by \cref{lovaszreplacement}, a contradiction.
    Hence, $X$ contains a vertex $a' \in A$ and a vertex $b' \in B$ such that $a'$ and $b'$ are not adjacent.
    \end{subproof}
    Throughout the rest of this proof, we fix $a', b'$ as in \cref{claim:non-adjacent-a-b-prime}.

    \begin{claim}
        $E \neq \emptyset$. \label{claim:E-not-empty}
    \end{claim}
    \begin{subproof}
    Suppose $E = \emptyset$.
    By \cref{claim:proper,claim:ab-perfect}, $(A,B)$ is proper and both $G[A]$ and $G[B]$ are $P_4$-free.

    We claim that each component $Q$ of $X$ induced by vertices in $A$ is a subpath of $X$ of even length. 
    Let $Q$ be a component of the subgraph of $X$ induced by $A$.
    Suppose $Q$ has odd length.
    Then since $G[A]$ is $P_4$-free, $Q$ consists of a single edge.
    Let $a_1$, $a_2$ be the vertices in $Q$.
    Since $N(A) \subseteq B \cup C$, it follows that 
    then there are two vertices $b_1, b_2 \in B \cap V(X)$ such that $a_1b_1$ and~$a_2b_2$ are both edges.
    Then $b_1$ and $b_2$ are non-adjacent because $X$ has length at least $5$. Then, for every $c\in C$, the vertices $c$,
    $a_1$, $a_2$, $b_1$, and $b_2$ induce a bull, a contradiction since $C \neq \emptyset$.
    Hence, every component of $G[V(X) \cap A]$ is a path of even length. By the symmetry between $A$ and $B$, every component of $G[V(X) \cap B]$ is a path of even length. 
        
    Suppose $X$ contains two non-adjacent vertices in~$A$.
    Then since each component of $G[X \cap V(A)]$ is an even-length path and $X$ has odd length, we can choose two non-adjacent $a_1, a_2 \in V(X) \cap A$ such that there exists an odd $a_1a_2$-subpath $P$ of $X$ whose internal vertices are not in~$A$.
    We denote the neighbor of $a_i$ in $P$ by $b_i$ for $i \in \{1,2\}$.
    Since $P$ is an odd path, $V(P) \cap C = \emptyset$ and $b_1, b_2$ are distinct vertices in $B$.
    Hence, $P$ contains an odd induced $b_1b_2$-path $\hat{P}$.
    Then, $\hat{P}$ cannot contain any vertex of $A \cup D$, so $\hat{P}$ is contained in $G[B]$.
    But $\hat{P}$ is a component of $G[V(X) \cap B]$, so it is a path of even length, a contradiction. 
(See \cref{fig:E-empty-easy} for an illustration.)
\begin{figure}
       \centering
       \begin{tikzpicture}
            \tikzstyle{v}=[circle, draw, solid, fill=black, inner sep=0pt, minimum width=3pt]
            \draw (0,1) node [v,label=left:$a_1$](a1){};
            \draw (0,0) node [v,label=left:$a_2$](a2){};
            \draw (1,1) node[v,label=below right:$b_1$](b1){};
            \draw (1,0) node[v,label=below:$b_2$](b2){};
            \draw (3,-1) node[v,label=below:$d\in D$](d){};
\draw (a1) to  (b1) [out=0,in=90] to (d);
            \draw (a2) to  (b2) [out=0,in=90] to (d);
            \draw [dotted] (b1)--(a2)--(a1)--(b2);
            \draw [dotted,out=45,in=60] (a1)to(d);
            \draw [dotted,out=-45,in=180](a2) to (d);
            \draw [decorate,decoration=snake] (b1)to (b2);
       \end{tikzpicture}
\caption{An illustration of the proof of \cref{claim:E-not-empty}.
       Non-edges are drawn as dotted lines. The wavy line between $b_1$ and $b_2$ indicates that $b_1$ and $b_2$ might be adjacent or they might be non-adjacent.
       If $b_1$ and $b_2$ are non-adjacent, $P$ contains some vertex $d \in D$, but then $P$ is not an induced odd path.
       If $b_1$ and $b_2$ are adjacent, $G$ contains a bull.
       }
       \label{fig:E-empty-easy}
   \end{figure}
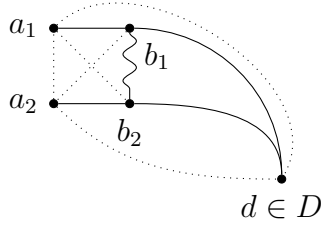
    Hence, $V(X) \cap A$ is a clique and thus $|V(X) \cap A| = 1$.
    By the symmetry between $A$ and $B$, it follows that $|V(X) \cap B| = 1$.
    So in particular, $a'$, $b'$ are the only vertices of $A \cup B$ in $X$.

    By \cref{claim:non-adjacent-a-b-prime}, $a'$ and $b'$ are not adjacent and therefore there is an $a'b'$-path $P$ of $X$ of even length in $H$ with interior in $H \setminus (A \cup B)$.
    Then, $H[V(P \setminus \{a', b'\}) \cup \{a, b\}]$ is an odd induced cycle of $H_i$.
    Hence, since $H_i$ contains no odd hole, $P$ has length two.
    But then $a$ and $b$ have a common neighbor in $V(G) \setminus (A \cup B)$ contrary to the assumption that $E = \emptyset$.
    \end{subproof}

    \Needspace{3\baselineskip}
    \begin{claim}
        One of $A$ and $B$ is a clique and the other is a stable set. \label{claim:E:clique-stable}
    \end{claim}
    \begin{subproof}
By~\cref{claim:E-not-empty}, $E$ is nonempty and therefore 
    $G[A \cup B]$ is perfect.
    Since $A\cup B$ is not a homogeneous set, $C\cup D$ is nonempty.    
    It follows from \ref{item:nonedges} and \ref{item:nonedges2} of \cref{lem:hompair} that $A$ or $B$ is a clique.
    Suppose both $G[A]$ and $G[B]$ contain an edge.
    Then by \ref{item:stable} and \ref{item:stable2} of \cref{lem:hompair}, $F$ is anti-complete to $C \cup D$ and $E$ is complete to $C \cup D$.
    Hence, $A \cup B \cup C \cup D$ is a homogeneous set in~$G$, a contradiction.
    \end{subproof}

    \begin{claim}
$\abs{V(X)\cap A}\le 1$ and $\abs{V(X)\cap B}\le 1$.
\label{claim:E-no-hole-with-two-in-A-or-two-in-B}
    \end{claim}
    \begin{subproof}
    Suppose $X$ contains two distinct vertices $a_1, a_2 \in A$.
    By~\cref{claim:non-adjacent-a-b-prime}, $\abs{V(X)\cap (A\cup B)}\ge3$ and so $V(X)\cap E=\emptyset$. Since the length of $X$ is at least $5$, we have $\abs{V(X)\cap C}\le 1$.
    Let $Q$ be a subpath of $X$ from $a_1$ to $a_2$ not containing any vertex of $C$.
    We choose $a_1$, $a_2$, and $Q$ such that the length of $Q$ is maximized.
    
    If $X$ has a vertex in $C$, then $\abs{E(Q)}=\abs{E(X)}-2\ge 3$.
    If $X$ has no vertex in $C$, then $\abs{E(Q)}\ge (\abs{E(X)}+1)/2\ge 3$.
    So, in both cases, $Q$ has length at least $3$.    

    Let $b_1$, $b_2$ be the neighbors of $a_1$, $a_2$ in $Q$, respectively.
    By~\cref{claim:E:clique-stable}, $b_1,b_2\notin A$ and so $b_1,b_2\in B$.  
    Since $Q$ is an induced path of $G$ with length at least $3$, 
    $b_1$ is non-adjacent to $a_2$
    and $b_2$ is non-adjacent to $a_1$.
    Then 
    $G[\{a_1,a_2,b_1,b_2\}]$ is isomorphic to $P_4$ by \cref{claim:E:clique-stable}, contradicting the assumptions that $G$ is gem-free and $E\neq\emptyset$ by \cref{claim:E-not-empty}.
By the symmetry between $A$ and~$B$, this completes the proof.
    \end{subproof}

Let $P$ be an $a'b'$-path of $X$.
Since each of $a'$ and $b'$ has exactly one neighbor in $V(P)$, 
$P$ does not contain more than one vertex of each of $C$, $D$, and $E$.
Since $X$ is not a hole of length~$4$, 
$X$ contains no more than one vertex of $E$.

\begin{claim}
    $V(X) \cap E = \emptyset$.
    \label{claim:X-no-E}
\end{claim}

\begin{subproof}
Suppose $X$ contains a vertex $v \in E$.
Let $P$ denote the path $X \setminus v$.
Then no interior vertex of $P$ is adjacent to~$v$, so none of the interior vertices of $P$ is complete to $E$.
Hence, no interior vertex of $P$ is in $A \cup B$.
By definition, $N(a') \subseteq A \cup B \cup C \cup E$ and $N(b') \subseteq A \cup B \cup C \cup E$ and $a', b' \in V(P)$.
It follows that $P$ contains a vertex in $C$ and a vertex in $D$.
In particular, neither $C$ nor $D$ can be complete to $E$, contradicting \cref{lem:hompair}\ref{item:completetoE}.
\end{subproof}

By~\cref{claim:E-no-hole-with-two-in-A-or-two-in-B,claim:X-no-E}, both $a'b'$-paths of $X$ have length at least three. 
Since one of the $a'b'$-paths of $X$ has even length,
there is an $a'b'$-path $P$ of $X$ of length at least four
and $P$ contains some vertex $c\in C$ and some vertex $d\in D$ by~\cref{claim:E-no-hole-with-two-in-A-or-two-in-B,claim:X-no-E}.
Now, $(V(P)\setminus\{a',b'\})\cup\{a,b\}$ induces an odd hole in $H_i$, a contradiction to the assumption that $H_i$ is perfect.
This completes the proof.
\end{proof}
Now we are ready to prove the main proposition of this subsection, which we restate here.
\elementarykperfect*
\begin{proof}
    Let $\gamma$ be the constant given by \cref{prop:gemfree-main} for $\F$.
    Note that $\gamma\ge 4$. 
    Let $G$ be an elementary bull-free locally perfect graph in $\F$. 
    By \cref{prop:gemfree-main}, if $G$ is gem-free, then $G$ is $\gamma$-perfect.
    Thus we may assume that $G$ has an induced subgraph $H$ that is a gem.
    Let $P$ be the path of length $3$ in $H$. 
    Then $V(P)$ is a dominating set of $G$ because $G$ is elementary. 
    Since $G$ is locally perfect, 
    $G[N_G(v)\cup \{v\}]$ is perfect for each $v\in V(P)$.
    Therefore, $G$ is $4$-perfect.
\end{proof}

\subsection{Completing the proof for bull-free graphs}

Previously, we defined elementary graphs, but for this subsection, we need to extend this notion to trigraphs. 
A trigraph $G$ is \emph{elementary} if it does not contain any path $P$ of length $3$ such that some vertex~$c$ of $V(G)\setminus V(P)$ is complete to $V(P)$
and some vertex~$a$ of $V(G)\setminus V(P)$ is anti-complete to $V(P)$.
We say $c$ is a \emph{center} for $P$ and $a$ is an \emph{anti-center} for $P$.

A \emph{hole} $H$ of length $5$ in a trigraph~$G$ is a subtrigraph of $G$ induced by $5$ vertices, say $h_1$, $h_2$, $h_3$, $h_4$, $h_5$ such that $h_{i}$ is adjacent to $h_{i+1}$ 
and anti-adjacent to $h_{i+2}$
for each $i\in\{1,2,\ldots,5\}$, assuming that $h_6=h_1$, $h_7=h_2$, $h_8=h_3$, and $h_9=h_4$.
For each $i\in\{1,2,\ldots,5\}$, 
\begin{itemize}
    \item let $L_i$ be the set of all vertices in $V(G)\setminus V(H)$ that are adjacent to $h_i$ and anti-complete to $V(H)\setminus\{h_i\}$,
    \item let $S_i$ be the set of all vertices in $V(G)\setminus V(H)$ that are anti-adjacent to $h_i$ and complete to $V(H)\setminus\{h_i\}$, and 
    \item let $C_i$ be the set of all vertices in $V(G)\setminus V(H)$ that are complete to $\{h_{i+1},h_{i+4}\}$ and anti-complete to $\{h_{i+2},h_{i+3}\}$.
\end{itemize}
A vertex in $L_i$, $S_i$, and $C_i$ is called a \emph{leaf}, a \emph{star}, a \emph{clone}, respectively, at $h_i$. A \emph{leaf}, a \emph{star}, or a \emph{clone} with respect to $H$ is a leaf, a star, or a clone, respectively, at $h_i$ for some $i\in\{1,2,\ldots,5\}$.

In~\cite{Chudnovsky2012a}, $\T_0$ is a precisely defined set of trigraphs and $\T_0$ is one of the base classes of trigraphs in the decomposition theorem of Chudnovsky~\cite{Chudnovsky2012}.
For our proof, we need only the following observation.

\begin{obs}\label{T0-eight-vertices}
    Every trigraph in $\T_0$ contains at most $8$ vertices. 
\end{obs}

The following theorem is a direct consequence of the proof of \cite[5.2]{Chudnovsky2012a}. The actual statement of \cite[5.2]{Chudnovsky2012a} is weaker in the sense that instead of (ii), \cite[5.2]{Chudnovsky2012a} deduces that one of $G$, $\overline G$ contains a ``homogeneous pair of type zero.'' 
It turns out that the only place in the proof deducing this consequence is the first sentence of the proof, which uses 4.1 of \cite{Chudnovsky2012a} to assume that there is no hole of length $5$ with both a leaf and a star. 
Thus, by removing the first sentence of the proof of 5.2 in Chudnovsky~\cite{Chudnovsky2012a}, we deduce the following slightly stronger statement.
\begin{theorem}[Chudnovsky~{\cite[5.2]{Chudnovsky2012a}}; strengthened form]\label{bulls1:52}
Let $G$ be a bull-free non-elementary trigraph. Then at least one of the following holds. 
\begin{enumerate}[label=\rm(\roman*)]
    \item $G$ or $\overline{G}$ belongs to $\T_0$.    
    \item $G$ has a homogeneous set.
    \item $G$ has a hole of length $5$ with both a leaf and a star.
\end{enumerate}
\end{theorem}

A trigraph is \emph{perfect} if 
every realization is perfect. 
We say a trigraph is \emph{imperfect} if it is not perfect.
Here is a corollary of \cref{lovaszreplacement} for trigraphs. 
\begin{lemma}\label{lovaszreplacement-trigraph}
    Let $A$ be a homogeneous set of a trigraph~$G$
    and $a\in A$.
    If both $G\setminus (A\setminus\{a\})$ and $G[A]$ are perfect, then $G$ is perfect.
    \qed 
\end{lemma}

A trigraph is \emph{$k$-perfect} if its vertex set can be partitioned into at most $k$ sets, each inducing a perfect trigraph.
We say a trigraph $G$ is \emph{locally perfect} if $G[N(v)]$ is perfect for every vertex~$v$ of $G$.
Then we obtain the following consequence of \cref{bulls1:52}.

\begin{lemma}\label{thm:neighbor-perfect-with-no-Q}
   Every locally perfect bull-free non-elementary graph is $2$-perfect, unless it has a hole of length $5$ with a leaf and a star.
\end{lemma}
\begin{proof}
    Suppose that $G$ is a locally perfect bull-free non-elementary graph that has no hole of length $5$ with a leaf and a star.
    We proceed by induction on $\abs{V(G)}$ to show that $G$ is $2$-perfect.
    We may assume that $G$ is connected and has more than $8$ vertices 
    because the disjoint union of two perfect graphs is perfect and every graph with at most four vertices is perfect.
    So by \cref{bulls1:52}, $G$ has a homogeneous set~$A \subseteq V(G)$.
    Moreover, there is some vertex~$v \in V(G) \setminus A$ that is complete to $A$ because $G$ is connected.
    Since $G$ is locally perfect, $G[A]$ is perfect.
    Let $a \in A$ and $G'=G\setminus (A \setminus \{a\})$.
    By the induction hypothesis, there is a partition of $V(G')$ into $X$, $Y$ such that $G'[X]$, $G'[Y]$ are both perfect.
    We may assume $a \in X$.
    We may assume that $X\neq\{a\}$ because otherwise $G[A]$ and $G\setminus A=G[Y]$ are perfect, implying that $G$ is $2$-perfect.

    Let $X'=X\cup A$ and let $G_X=G[X']$.
    Note that both $G_X\setminus (A\setminus\{a\})=G'[X]$ and $G_X[A]=G[A]$ are  perfect and $A$ is a homogeneous set of $G_X$. 
    By \cref{lovaszreplacement}, $G_X$ is perfect.
    So $(X', Y)$ is a partition of $V(G)$ such that both $G[X']$ and $G[Y]$ are perfect.
\end{proof}

The following theorem is a direct consequence of the proof of 4.3 in \cite{Chudnovsky2012a}.

\begin{theorem}[Chudnovsky~{\cite[4.3]{Chudnovsky2012a}}; weaker but more detailed form]\label{bull-free-with-pentameter}
    Let $G$ be a bull-free trigraph satisfying the following properties.
    \begin{itemize}
        \item Neither $G$ nor $\overline G$ belongs to $\mathcal T_0$.
        \item $G$ has a hole $H$ of length $5$ induced by $5$ vertices $h_1,h_2,h_3,h_4,h_5$ in this order and $H$ has both a star at $h_1$ and a leaf at $h_1$.
        \item $G$ has no homogeneous set.
    \end{itemize}
    Then $G$ has a tame homogeneous pair $(A,B)$ with the following properties, where $C_i$ denotes the set of clones at $h_i$ for all $i\in\{1,2,\ldots,5\}$.
\begin{enumerate}[label=\rm(\roman*)]
    \item $A = \{h_2, h_5 \} \cup C_2 \cup C_5$.
    \item $B = \{h_3, h_4 \} \cup C_3 \cup C_4$.
    \item There is a vertex $v \in V(G) \setminus (A \cup B)$ strongly complete to $A \cup B$.
\end{enumerate}
\end{theorem}

We say that a trigraph is \emph{austere} if
\begin{enumerate}[label=\rm(\alph*)]
    \item\label{item:austere-mon} it is monogamous,
    \item\label{item:austere-set} no homogeneous set contains a switchable pair, and 
    \item\label{item:austere-pair} for every dominated tame homogeneous pair $(A,B)$, $A\cup B$ contains no switchable pair.
\end{enumerate} 

\begin{lemma}\label{lem:austere-set}
    Let $G$ be an austere trigraph.
    If $A$ is a homogeneous set of~$G$ and $a\in A$, then
    $G\setminus (A\setminus\{a\})$ is also austere.
\end{lemma}
\begin{proof}
    Let $G'=G\setminus (A\setminus\{a\})$.
    Clearly, $G'$ satisfies \ref{item:austere-mon}.

    To prove \ref{item:austere-set}, suppose that $G'$ has a homogeneous set $X$.
    If $a\notin X$, then $X$ is also a homogeneous set of $G$ and so $X$ contains no switchable pair in $G'$. 
    If $a\in X$, then $A\cup (X\setminus\{a\})$ is a homogeneous set of $G$ and so $A\cup (X\setminus \{a\})$ contains no switchable pair in $G$. 
    This means that $X$ contains no switchable pair in $G'$. This proves \ref{item:austere-set}.

    For \ref{item:austere-pair}, suppose that $G'$ has a dominated tame homogeneous pair $(X,Y)$.
    If $a\notin X\cup Y$, then $(X,Y)$ is a dominated tame homogeneous pair of $G$ and therefore $X\cup Y$ has no switchable pair in both $G$ and $G'$. 
    If $a \in X \cup Y$, then 
    we may assume $a \in X$.
By definition of a homogeneous set, $(A \cup (X \setminus \{a\}), Y)$  is a dominated tame homogeneous pair in~$G$. Hence, 
    $A \cup (X \setminus \{a\})\cup Y$ contains no switchable pairs in $G$
    and so $X\cup Y$ contains no switchable pair in $G'$.
\end{proof}
\begin{lemma}\label{lem:austere-hom}
    Let $G$ be an austere trigraph and $(A,B)$ be a maximal dominated tame homogeneous pair of~$G$.
    If $A \cup B$ is not a subset of any homogeneous set of $G$,
    then the trigraph obtained by shrinking $(A,B)$ is also austere.
\end{lemma}
\begin{proof}
    Let $G'$ be the trigraph obtained by shrinking $(A, B)$ and let $a, b$ be the vertices of $G'$ corresponding to $A$ and $B$, respectively.

    By the definition of a homogeneous pair, the only switchable pair containing $a$ or $b$ in $G'$ is the pair $\{a,b\}$. Hence, $G'$ is monogamous because $G$ is monogamous. This proves \ref{item:austere-mon}.
    
    For \ref{item:austere-set}, suppose that $G'$ has a homogeneous set $X$ that contains a switchable pair.
    Then since $G$ is austere, $X$ is not a homogeneous set in $G$.
    Hence, $X$ contains $a$ or $b$ and so by the definition of a homogeneous set, $X$ contains both $a$ and $b$.
    But then $A \cup B\cup (X\setminus \{a, b\})$ is a homogeneous set of $G$, contradicting our choice of $(A, B)$.
    This proves \ref{item:austere-set}.

    For \ref{item:austere-pair}, suppose that $G'$ has a dominated tame homogeneous pair $(X,Y)$ such that $X \cup Y$ contains a switchable pair in~$G'$.
    Then, $X \cup Y$ contains $a$ or $b$.
    Since $\{a,b\}$ is a switchable pair, by definition of a homogeneous pair, $X\cup Y$ contains both $a$ and $b$.
    Then if both $a, b \in X$, the $(A \cup B \cup (X \setminus \{a,b\}), Y)$ is a dominated tame homogeneous pair of $G$ and it properly contains $(A, B)$, a contradiction. 
    Hence, we may assume $a \in X$ and $b \in Y$.
    Then, $(A \cup (X \setminus \{a\}), B\cup  (Y \setminus \{b\}))$ is a dominated tame homogeneous pair of $G$ and it properly contains $(A, B)$, a contradiction.
    This proves \ref{item:austere-pair}.   
\end{proof}
\begin{proposition}\label{prop:austere}
For every $2$-good class~$\F$ of graphs,
    there exists an integer~$c_\F$ satisfying the following. 
    \begin{itemize}
        \item[] 
            For every locally perfect bull-free austere trigraph~$G$ whose every induced subtrigraph without switchable pairs is in $\F$, 
            there exists a partition $(X_1,X_2,\ldots,X_k)$ of $V(G)$ with $k\le c_\F$ such that 
            $G[X_i]$ is a perfect subtrigraph with no switchable pair
            for all $i\in\{1,2,\ldots,k\}$.
    \end{itemize}
\end{proposition}
\begin{proof}
    Let $c_\F=2\gamma\ge 2$ where $\gamma$ is defined in \cref{elementary-bull-free-constant-perfect} for $\F$.
    We proceed by the induction on $\abs{V(G)}$.
    As every trigraph on at most $4$ vertices is perfect, we may assume that $\abs{V(G)}>8$
    and therefore neither $G$ nor $\overline G$ belongs to $\mathcal T_0$.
    Since the disjoint union of two perfect trigraphs is perfect, we may assume that $G$ is connected.

    Since $G$ is monogamous, there exists a partition $(S,T)$ of $V(G)$ such that both $G[S]$ and $G[T]$ have no switchable pairs.     
    So both $G[S]$ and $G[T]$ are locally perfect bull-free elementary graphs. 
    Suppose that $G$ is elementary. 
    By applying \cref{elementary-bull-free-constant-perfect} to both $G[S]$ and $G[T]$, we obtain a partition of $V(G)$ into at most $2\gamma$ subsets, each inducing a perfect induced subtrigraph without switchable pairs.
    Therefore we may assume that $G$ is not elementary.

    Suppose that $G$ has a homogeneous set $A$. 
    Let $a\in A$ and $G'=G\setminus (A\setminus\{a\})$.
Then trivially, $G'$ is locally perfect and bull-free.  
    By \cref{lem:austere-set}, $G'$ is austere.
    By the induction hypothesis, $G'$ admits 
    a partition $(X_1,\ldots,X_{k})$ of $V(G')$ 
    with $k\le c_\F$ 
    such that $G'[X_i]$ is perfect and has no switchable pair for each $i\in\{1,2,\ldots,k\}$.
    We may assume that $a\in X_1$. 
Since $G$ is connected and $A$ is a homogeneous set of~$G$, there is a vertex $v\in V(G)$ such that $v$ is strongly complete to $A$. 
    Since $G$ is locally perfect, $G[A]$ is perfect.
    By \cref{lovaszreplacement-trigraph}, 
$G[X_1\cup A]$
    is still perfect.
    Furthermore, $G[X_1\cup A]$ has no switchable pair because both $G[A]$ and $G[X_1]$ have no switchable pair.
Then $(X_1\cup A,X_2,\ldots,X_{k})$ is a desired partition of $V(G)$.
    Thus, we may assume that $G$ has no homogeneous set.

    By \cref{bulls1:52}, $G$ has a hole $H$ of length $5$ with both a star and a leaf.
    By \cref{bull-free-with-pentameter}, $G$ has 
    a dominated tame homogeneous pair.
    Thus, there exists a maximal dominated tame homogeneous pair $(A,B)$.
    Since $G$ is locally perfect and $(A,B)$ is dominated, both $G[A]$ and $G[B]$ are perfect.

    Let $G_0$ be the trigraph obtained from $G$ by shrinking $(A,B)$.
    Observe that 
    every realization of $G_0$ is isomorphic to an induced subgraph of some realization of $G$.
This implies that $G_0$ is bull-free and locally perfect.

    Let $a$, $b$ be the vertices of $G_0$ corresponding to $A$, $B$, respectively. 
    By the induction hypothesis, 
    $G_0$ admits 
    a partition $(X_1,\ldots,X_{k})$ of $V(G_0)$ with $k\le c_\F$  
    such that $G_0[X_i]$ is perfect and has no switchable pair for each $i\in\{1,\ldots,k\}$.
    We may assume that $a\in X_1$ and $b\in X_2$
    because no $X_i$ contains switchable pairs.

    Let $X_1'=(X_1\setminus\{a\})\cup A$ and 
    $X_2'=(X_2\setminus \{b\})\cup B$.
    By \cref{lovaszreplacement-trigraph}, both $G[X_1']$ and $G[X_2']$ are perfect. 
    Furthermore, both $G[X_1']$ and $G[X_2']$ have no switchable pairs because $G$ is austere.
    Observe that for all $i\in\{3,\ldots,k\}$, $G[X_i]=G'[X_i]$. 
    Therefore $(X_1',X_2',X_3,\ldots,X_{k})$ is the desired partition of $V(G)$. 
\end{proof}

Since every graph is also an austere trigraph, we obtain \cref{thm:non-elementary-locally-perfect} as a direct corollary to \cref{prop:austere}.
Recall this implies the class of bull-free graphs is $2$-strongly Pollyanna by \cref{thm:ets-locally-perfect-basic}
\footnote{In fact our proof gives a slightly stronger statement; It is not difficult to see that our proof implies that for any $2$-good class $\mathcal{F}$ every bull-free $G$ in $\mathcal{F}$ satisfies $\chi(G) \leq \omega(G)^{\mathcal{O}(\tau)}$ where $\tau = \chi^{(2)}(\mathcal{F})$ is the minimum integer such that every triangle-free graph in $\F$ is $\tau$-colorable.
The proof of \cref{thm:non-elementary-locally-perfect} shows that $c_{\mathcal{F}} \leq 2\gamma \leq  \max\{12, 2\tau + 2\}$.
The argument of \cref{thm:ets-locally-perfect-basic}, shows that every \emph{basic} bull-free graph in $\mathcal{F}$ is $\omega^{c_{\mathcal{F}} +1}$-colorable.
As shown by Bourneuf and Thomass{\'{e}} \cite{bourneuf2023bounded}, \cref{thm:substitution-orig} can be improved to say that if a hereditary graph class $\mathcal{G}$ is $\chi$-bounded by $\omega^k$ then its closure under substitution is $\chi$-bounded by $\omega^{2k+3}$.
So by applying this to \cref{every-composite-has-a-hom-set}, we obtain that 
every bull-free $G \in \mathcal{F}$ satisfies $\chi(G) \leq \max 
\{\omega(G)^{4\tau + 7}, \omega(G)^{29}\}$.}.
We restate \cref{thm:non-elementary-locally-perfect} for the convenience of the reader.
\nonelementary*

\section{Non-Pollyanna classes}\label{sec:nonpollyanna}

A \emph{oriented tree} is an orientation of a tree.
For a positive integer $n$, a graph $G$ is an \emph{$n$-willow} 
if there exists an oriented tree $T$ with $V(G)\subseteq V(T)$ such that 
for every distinct pair $u$, $v$ of vertices of $G$, the vertices
$u$ and $v$ are adjacent if and only if 
$T$ has a directed path from~$u$ to~$v$ or from~$v$ to~$u$ whose length is not a multiple of~$n$.
In this case, we say $G$ is an $n$-willow defined by~$T$.
We will make extensive use of the following easy observation.
\begin{obs}\label{obs:multipartite}
    Let $n$ be a positive integer and let $T$ be an oriented tree. 
    If $P$ is a directed path in~$T$ and $G$ is an $n$-willow defined by~$T$, 
    then $G[V(P)\cap V(G)]$ is a complete multipartite graph.
\end{obs}
A graph is a \emph{willow} if it is an $n$-willow for some positive integer $n$. We remark that by subdividing certain edges of the associated oriented tree, one can show that if a graph is an $n$-willow, then it is also an $n'$-willow for all $n'\ge n$.
On the other hand, the clique number of an $n$-willow is at most $n$ and $K_n$ is an $n$-willow, so for every positive integer $n\ge 2$, there are $n$-willows that are not $n'$-willows for any positive integer $n'<n$.

The main result of this section is the following theorem which relates willows and Pollyanna classes of graphs.
\begin{theorem}\label{non-pollyanna}
    If $\mathcal{F}$ is a finite set of graphs, none of which is a willow, then the class of $\F$-free graphs is not Pollyanna.
\end{theorem}
To construct $\chi$-bounded hereditary classes of graphs that are not polynomially $\chi$-bounded, 
we need the following lemmas,  explicitly stated in 
Bria{\'n}ski, Davies, and Walczak~{\cite{BDW2022}},
originated from \cite{CARBONERO202363} and \cite{KT1992}.

\begin{lemma}[Bria{\'n}ski, Davies, and Walczak~{\cite[Lemma 4]{BDW2022}}]\label{non-poly2}
    Let $k$ be a positive integer. Then, there is a graph $G$ with an acyclic orientation of its edges satisfying the following.
    \begin{enumerate}[label=\rm(A\arabic*)]
\item \label{A:colors} $\chi(G)=k$.
\item \label{A:unique} For every pair of vertices $u$ and $v$, there is at most one directed path from $u$ to~$v$ in $G$.
\item \label{A:path} There is a directed path in $G$ on $k$ vertices.
\item \label{A:coloring} There is a $k$-coloring $\phi$ of $G$ such that 
for every directed path in $G$ of non-zero length, their ends $u$ and $v$ satisfy that $\phi(u)\neq\phi(v)$.
\end{enumerate}
\end{lemma}

\begin{lemma}[Bria{\'n}ski, Davies, and Walczak~{\cite[Lemmas 5 and 6]{BDW2022}}]\label{non-poly}
Let $p\le k$ be positive integers with $p$ prime, and let $G$ be a graph with an acyclic orientation of its edges satisfying \ref{A:colors}, \ref{A:unique}, \ref{A:path}, and \ref{A:coloring} for $k$.
Let $G_p$ be the graph obtained from $G$ by adding an edge $uv$ 
whenever $G$ has a directed path between $u$ and $v$ whose length is not divisible by $p$.
Then, $\omega(G_p)=p$ and every induced subgraph of $G$ with clique number $m<p$ has chromatic number at most ${m+2 \choose 3}$.
\end{lemma}

Graphs $G$ as in \cref{non-poly2} exist, and Bria{\'n}ski, Davies, and Walczak~\cite{BDW2022} showed specifically that the natural orientation of Tutte's construction~\cite{descartes1947three,Descartes1954} has these properties.
Note that \ref{A:colors} implies \ref{A:path} by the following well-known lemma
due to Gallai~\cite{Gallai1968}, Hasse~\cite{Hasse1964}, Roy~\cite{Roy1967}, and Vitaver~\cite{Vitaver1962}.
\begin{lemma}[Gallai, Hasse, Roy, and Vitaver \cite{Gallai1968,Hasse1964,Roy1967,Vitaver1962}]\label{lem:path-colors}
Let $k$ be a positive integer.
    If a graph $G$ has an orientation with no directed path of length $k$,  
    then $\chi(G) \leq k$.
\end{lemma}

Gir{\~a}o, Illingworth, Powierski, Savery, Scott, Tamitegama, and Tan~\cite{GIPSSTT2022} considered the construction of Nešetřil and Rödl~\cite{NESETRIL1979225}, which is a large-girth variation of the construction of Tutte~\cite{descartes1947three,Descartes1954}. Using the same natural orientation, they obtained the following.

\begin{lemma}[Gir{\~a}o, Illingworth, Powierski, Savery, Scott, Tamitegama, and Tan~{\cite[Lemma 10]{GIPSSTT2022}}]\label{lem:jane-and-friends}
    For every $g \ge 3$ and $k\ge 2$, there is a graph $Y$ with an orientation of its edges such that $\chi(Y)=k$ and every cycle in $Y$ contains at least $g$ changes of direction in the orientation.
\end{lemma}

The property \ref{A:coloring} also clearly holds for this construction, since the same natural orientation and coloring from the proof of Bria{\'n}ski, Davies, and Walczak~{\cite{BDW2022}} for the construction of Tutte~\cite{descartes1947three,Descartes1954} can be used.
Note that the orientation of $Y$ described in 
\cref{lem:jane-and-friends} is acyclic and satisfies \ref{A:unique} because all of its cycles have at least three changes in direction in the orientation.
By \cref{lem:path-colors}, \ref{A:path} holds for $Y$.
Thus, we obtain the following strengthening of \cref{non-poly2}.

\begin{lemma}\label{girth2}
    Let $g,k$ be positive integers with $g \ge 3$ and $k\ge 2$. Then, there is a graph $G$ with an orientation of its edges satisfying \ref{A:colors}, \ref{A:unique}, \ref{A:path}, and \ref{A:coloring} for $k$ and additionally:
    
\begin{enumerate}[label=\rm(B\arabic*)]
\item every cycle in $G$ contains at least $g$ changes of direction in the orientation.
\label{B:1}
\qed
\end{enumerate}
\end{lemma}

\begin{lemma}\label{lem:is-willow}
    Let $g$, $k$ be positive integers with $g\ge 3$ and $k\ge 2$. 
    Let $p$ be a prime less than or equal to $k$. 
    Let $G$ be a graph with an orientation of its edges satisfying \ref{A:colors}, \ref{A:unique}, \ref{A:path}, and \ref{A:coloring} for $k$ and \ref{B:1} for~$g$. 
    Let $G'$ be the graph on $V(G)$ such that two vertices $u$, $v$ are adjacent in $G'$ if and only if there is a directed path between $u$ and $v$ whose length is not divisible by $p$.
    If $g> \binom{N}{2}$ for an integer $N$, then every induced subgraph of $G'$ with at most $N$ vertices is a $p$-willow.
\end{lemma}
\begin{proof}
    Let $X$ be a set of at most $N$ vertices of $G'$. 
    We claim that $G'[X]$ is a $p$-willow. 
    Let $T$ be the union of all directed paths of $G$ between $u$ and $v$ whose length is not divisible by $p$ for all edges $uv$ of $G'[X]$.

    By \ref{A:unique}, we added at most $1$ directed path per every edge of $G'[X]$ and therefore in total $T$ consists of fewer than $g$ directed paths.
    By \ref{B:1}, every cycle in $G$ contains at least $g$ changes of direction 
    and therefore $T$ has no cycles. 
    Let $T'$ be a tree obtained from $T$ by adding a new vertex with an out-edge to one vertex of each component of $T$.
    Then $T'$ is a tree.     

    Observe that for distinct vertices $u$ and $v$ in $X$, if $T'$ has a directed path from $u$ to~$v$ whose length is not a multiple of $p$, 
    then so does $G$ and therefore 
    $G'$ contains the edge $uv$
    by the definition of $G'$.
    Conversely, if $G'[X]$ contains an edge $uv$, then $G$ contains a directed path $P$ between $u$ and $v$ whose length is not a multiple of $p$. 
    By \ref{A:unique}, such a path $P$ is unique and therefore $T'$ contains $P$.
    This proves that $G'[X]$ is a $p$-willow defined by $T'$.
\end{proof}

Now we can prove \cref{non-pollyanna}. We obtain a $\chi$-bounded class that is not polynomially $\chi$-bounded by combining \cref{non-poly} with \cref{girth2} for some suitably large~$g$ instead of \cref{non-poly2} as is done in \cite{BDW2022}. Then, it is just a matter of examining the induced subgraphs.

\begin{proof}[Proof of Theorem \ref{non-pollyanna}.]
    Let $\mathbb{N}$ be the set of positive integers.
    Let $N$ be the maximum number of vertices of a graph in $\mathcal{F}$ and let $g=\max(\binom{N}{2} +1,3)$.
    Choose a function $f:\mathbb N\to\mathbb N$ such that $f(1)=1$, $f(n)\ge \binom{n+2}{3}$ for all $n\in \mathbb{N}$, and $\lim_{n \to \infty} \frac{f(n)}{n^k} = \infty$ for every positive integer~$k$.
    In other words, we choose $f$ to be \say{superpolynomial}.

    Let us first construct a $\chi$-bounded class $\mathcal Z$ of graphs that is not polynomially $\chi$-bounded.
    For each prime $p$, let $Y_p$ be a graph with an orientation of its edges satisfying \ref{A:colors}--\ref{A:coloring} for $k:=f(p)$ and \ref{B:1} for~$g$, given by \cref{girth2}.
    For every prime $p$, we define $E_p$ to be the set consisting of all pairs~$\{u, v\}$ where $u, v \in V(Y_p)$ and $Y_p$ contains a directed path from $u$ to~$v$ or from $v$ to $u$ whose length is not divisible by $p$.
    Let $Z_p$ be the graph $(V(Y_p), E_p)$.
    Note that $E(Y_p)\subseteq E_p$.
    In other words, $Z_p$ can be obtained from $Y_p$ by adding the elements of $E_p$ to the edge set of $Y_p$.

    By \cref{non-poly}, we have that $\omega(Z_p) = p$ and 
    every induced subgraph $Z$ of $Z_p$ with clique number $m<p$ has chromatic number at most $\binom{m+2}{3}$.
    By \ref{A:colors} and \ref{A:coloring}, $\chi(Z_p)=k=f(p)$. 
    Let $\hat{\mathcal{Z}}$ be the set of all graphs $Z_p$ for each prime $p$ and let $\mathcal{Z}$ be the closure of $\hat{\mathcal{Z}}$ under taking induced subgraphs.
    Then $\mathcal{Z}$ is $\chi$-bounded by a $\chi$-bounding function~$f$.
    Since there are infinitely many primes and for every prime $p$ there is a graph $Z\in \mathcal{Z}$ with clique number $p$ and chromatic number $f(p)$, $\mathcal{Z}$ is not polynomially $\chi$-bounded by our choice of $f$.

    Now, suppose that the class $\mathcal C$ of $\mathcal F$-free graphs is Pollyanna. 
    Then $\mathcal Z\not\subseteq \mathcal C$ because $\mathcal Z$ is not polynomially $\chi$-bounded. 
    Then there exist a prime $p$ and a set $X \subseteq V(Z_p)$ such that $Z_p[X]$ is isomorphic to a graph $F\in\mathcal F$.
By \cref{lem:is-willow}, $Z_p[X]$ is a $p$-willow, contradicting the assumption that $\F$ contains no willows.
\end{proof}

We remark that by applying \cref{girth2,non-poly,lem:is-willow}, one can also obtain the following.

\begin{theorem}\label{non-nearchi}
    If $\mathcal{F}$ is a finite set of graphs, none of which is a willow, then for every positive integer~$q$, there is a class $\mathcal{G}$ of $\F$-free graphs that is not $\chi$-bounded, but such that every graph $G\in \mathcal{G}$ with $\omega(G)< q$ has chromatic number at most $\binom{q+1}{3}$.
\end{theorem}
\begin{proof}
    Let $p$ be a prime such that $q\le p\le 2q$ (such a prime exists by Bertrand’s postulate).
    Let $N$ be the maximum number of vertices of a graph in $\F$
    and let $g=\max(\binom{N}{2}+1,3)$.

    For each integer $k\ge p$, we are going to construct a graph $G_k$ as follows.
    By~\cref{girth2}, there is a graph $H_k$ with an orientation of its edges satisfying \ref{A:colors}--\ref{A:coloring} for $k$ and \ref{B:1} for~$g$.
    By~\cref{non-poly}, there is a graph $G_k$ obtained from $H_k$ by adding an edge $uv$ whenever $H_k$ has a directed path between $u$ and $v$ whose length is not divisible by $p$
    such that 
    $\omega(G_k)=p$ 
    and every induced subgraph of $G_k$ with clique number $m<p$ has chromatic number at most $\binom{m+2}{3}$.
    By~\labelcref{A:colors,A:coloring}, $\chi(G_k)=k$.
    Let $\G$ be the class of all induced subgraphs of $G_k$ for all~$k\ge p$.
    So, $\G$ is not $\chi$-bounded but every graph in $\G$ with $\omega(G)=m<q$ has chromatic number at most $\binom{m+2}{3}\le \binom{q+1}{3}$.

    By \cref{lem:is-willow}, every graph in $\G$ with at most $N$ vertices is a $p$-willow and therefore $\G$ is $\F$-free.    
\end{proof}

\section{Forbidden induced subgraphs for willows} \label{sec:willows}

In this section, we describe some forbidden induced subgraphs for the class of willows.
We only aim to sample the forbidden induced subgraphs rather than to find an exhaustive list.
We believe there are many more.
Our main idea is to use \cref{obs:multipartite}, which says that if $G$ is an $n$-willow defined by an oriented tree $T$, then vertices on a directed path on $T$ cannot induce $K_2\cup K_1$ in $G$, because $K_2\cup K_1$ is not a complete multipartite graph.

A $10$-vertex graph $G$ is 
a \emph{pentagram spider} if it has a perfect matching $M$ such that  $G\setminus M$ has a component isomorphic to $K_5$.
Note that vertices not in the component isomorphic to~$K_5$ are allowed to be adjacent to each other.
See \cref{fig:willow-forbids} for an illustration.

\begin{proposition}\label{pentagram spider}
    No pentagram spider is a willow.
\end{proposition}

\begin{proof}
    Let $G$ be a pentagram spider
    and $M$ be a perfect matching of $G$ such that $G\setminus M$ has a clique $A$ of size $5$.
    Let $T$ be an oriented tree and suppose that $G$ is a willow defined by $T$.
   Then by definition $V(G) \subseteq V(T)$ and for every edge $uv\in E(G)$, there is a directed path from $u$ to~$v$ or from $v$ to $u$ in~$T$.
    Since $A$ is a clique of $G$, there is a directed path $P$ in~$T$ which contains all vertices of $A$. 
    Let $x_1,x_2,x_3,x_4,x_5$ be the vertices of $A$ 
    in the order of their appearances in $P$.
    Let $y_1,y_2,y_3,y_4,y_5$ be the vertices of $G$ such that $x_iy_i\in M$ for all $i=1,2,\ldots,5$.
    Since $x_3y_3\in E(G)$, there is some directed path $P'$ in~$T$ from $y_3$ to $x_3$ or from $x_3$ to $y_3$.
    By reversing the orientation of all edges of $G$ and $T$ and switching the labels of $x_1$, $x_2$ with $x_5$, $x_4$ if necessary, we may assume that $P'$ is a directed path from $y_3$ to $x_3$.
    Then, there is a directed path $P''$ in~$T$ containing $y_3,x_3,x_4,x_5$ in order.
    Then $G[\{y_3,x_4,x_5\}]$ is not a complete multipartite graph, contradicting \cref{obs:multipartite}.
\end{proof}

A $12$-vertex graph is a \emph{tall strider} if it has a clique $C=\{x_1,x_2,x_3\}$ of size $3$ such that 
$N(x_1)\setminus C$, $N(x_2)\setminus C$, and $N(x_3)\setminus C$ are disjoint cliques of size $3$. 
We remark that there can be edges between $N(x_i)\setminus C$ and $N(x_j)\setminus C$ for distinct $i$, $j$.
See \cref{fig:willow-forbids} for an illustration. 

\begin{proposition}\label{tall strider}
    No tall strider is a willow.
\end{proposition}

\begin{proof}
    Let $G$ be a tall strider with a clique $C$ of size $3$ such that $N(v)\setminus C$ for all $v\in C$ are disjoint cliques of size $3$.
    Let $T$ be an oriented tree and suppose that $G$ is a willow defined by~$T$.
    Since $C$ is a clique of $G$, there is a directed path $P$ in~$T$ that contains all vertices of~$C$. 
    Let $x_1$, $x_2$, $x_3$ be the vertices in $C$ such that $P$ is a directed path from $x_1$ to~$x_3$.
    Similarly, since $(N(x_2)\setminus C)\cup\{x_2\}$ is a clique, there exists a directed path $P'$ in~$T$ that contains all vertices of 
    $(N(x_2)\setminus C)\cup\{x_2\}$.
    If two vertices, say $a$, $b$ of $N(x_2)\setminus C$ come after~$x_2$ in $P'$, then $T$ contains a directed path containing $x_1$, $x_2$, $a$, and $b$. However, $G[\{x_1,a,b\}]$ is not a complete multipartite graph, contradicting~\cref{obs:multipartite}.
    Thus two vertices, say $a$, $b$ of $N(x_2)\setminus C$ come before~$x_2$ in $P'$. 
    Then $T$ contains a directed path containing $a$, $b$, $x_2$, $x_3$. Again, $G[\{a,b,x_3\}]$ is not a complete multipartite graph, contradicting~\cref{obs:multipartite}.
\end{proof}

A $10$-vertex graph is 
a \emph{short strider} if it has a clique $C=\{x_1,x_2,x_3,x_4\}$ of size $4$ such that 
$N(x_1)\setminus C$, $N(x_2)\setminus C$, and $N(x_3)\setminus C$ are disjoint cliques of size $2$.
We remark that there can be edges between $N(x_i)\setminus C$ and $N(x_j)\setminus C$ for distinct $i$, $j$.
See \cref{fig:willow-forbids} for an illustration.

\Needspace{3\baselineskip}
\begin{proposition}\label{short strider}
    No short strider is a willow.
\end{proposition}

\begin{proof}
    Let $G$ be a short strider.
    Let $T$ be an oriented tree and suppose that $G$ is a willow defined by $T$.
    Let $C=\{x_1,x_2,x_3,x_4\}$ be a clique of $G$ such that 
    $N(x_1)\setminus C$, $N(x_2)\setminus C$, and $N(x_3)\setminus C$ are disjoint cliques of size $2$.

    Since $C$ is a clique of $G$, we may assume without loss of generality that $T$ has a directed path $P$ that contains all vertices in $C$. 
    By reversing the direction of all edges in~$T$ if necessary, we may assume $x_4$ is not the first two vertices of $C$ in $P$.
    By the symmetry among $x_1$, $x_2$, and~$x_3$, we may assume that 
    $x_1$ is the first vertex of $C$ appearing on $P$
    and $x_2$ is the second vertex of $C$ appearing on $P$. 
    Since $(N(x_2)\setminus C)\cup \{x_2\}$ is a clique of $G$, there is a directed path~$P'$ in~$T$ that contains all vertices in $(N(x_2)\setminus C)\cup \{x_2\}$.
    
    If some $x\in N(x_2)\setminus C$ appears before $x_2$ on $P'$, then $T$ has a directed path $P''$ containing $x$, $x_2$, $x_3$, and $x_4$. 
    However, $G[\{x,x_3,x_4\}]$ is not a complete multipartite graph, contradicting~\cref{obs:multipartite}.

    We may therefore assume that two vertices in $N(x_2)\setminus C$ appear after $x_1$ on $P'$.
    But then, $T$ has a directed path $P^*$ containing $x_1$, $x_2$ and two vertices in $N(x_2)\setminus C$.
    Then $G[\{x_1\}\cup (N(x_2)\setminus C)]$ is not a complete multipartite graph, contradicting~\cref{obs:multipartite}.
\end{proof}
\begin{figure}
    \centering
    \begin{tikzpicture}[xscale=.8,
        every edge quotes/.style = {auto, font=\footnotesize, sloped
        }, 
        every edge/.style={->,dashed,-{Latex[length=2mm]},draw}]
        \tikzstyle{v}=[circle, draw, solid, fill=black, inner sep=0pt, minimum width=3pt]

        \node at (0,0)[v,label=below:$v_2$] (v5) {};
        \node at (2,0)[v,label=below:$v_4$] (v4) {};
        \node at (4,0)[v,label=below:$v_1$] (v3) {};
        \node at (6,0)[v,label=below:$v_8$] (v2) {};
        \node at (8,0)[v,label=below:$v_5$] (v1) {};
        \node at (10,0)[v,label=below:$v_7$] (v0) {};
        \node at (3,1)[v,label=left:$v_3$] (t) {};
        \node at (7,1)[v,label=right:$v_6$] (b) {};
        \draw (v5) edge["$n-1$"] (v4);
        \draw (v4) edge["$1$"](v3);
        \draw (v3) edge["$n-2$"] (v2);
        \draw (v2) edge["$1$"] (v1); 
        \draw (v1) edge["$n-1$"](v0);
        \draw (t) edge["$3$"](v3);
        \draw (v2) edge["$3$"] (b);
    \end{tikzpicture}
    \caption{The complement $\overline{P_8}$ of $P_8$ is an $n$-willow for every integer $n\ge 5$. Vertices $v_1$, $v_2$, $\ldots$, $v_8$ represent vertices of $\overline{P_8}$ in the order.
    The dashed arc with an integer $k$ means a directed path of length $k$.}\label{fig:P8}
\end{figure}
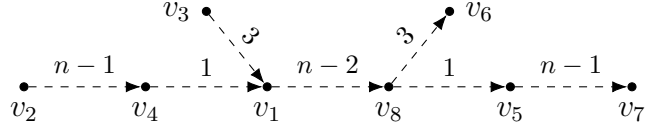

Now we present a lemma on willows, which we will use in later propositions.

\begin{lemma}\label{claim:length-2-path}
    Let $G$ be a graph whose complement $\overline{G}$ is a willow defined by an oriented tree $T$. 
    If $G$ has an induced path $u\dd v\dd w$ of length $2$, then 
    $T$ has no directed path between $u$ and $v$
    \emph{or}
    $T$ has no directed path between $v$ and $w$.
\end{lemma}
\begin{proof}
    Suppose not. 
    Then, without loss of generality, we may assume that there exists a directed path $P$ between $u$ and $v$ in $T$.
    By reversing all edges of $T$ if necessary, we may assume $P$ is a directed path from $u$ to $v$.
    Observe that $\overline{G}[\{u,v,w\}]$ is isomorphic to $K_2\cup K_1$.
    Since $K_2 \cup K_1$ is not a complete multipartite graph by~\cref{obs:multipartite}, it follows that there is no directed path from $v$ to $w$. Therefore, there exists a directed path from $w$ to~$v$ in~$T$.
    Since $T$ is a tree, it now follows that $T$ has no directed path between $u$ and $w$, contradicting the fact that $uw\in E(\overline{G})$.
\end{proof}

We remark that $\overline{P_8}$ is a willow, see \cref{fig:P8}. 
Next, we show that $\overline{P_{9}}$ is not a willow.
This clearly follows from the following more general proposition.

\begin{proposition}\label{P9}
    Let $G$ be a graph. 
    If $G$ has three vertex-disjoint induced paths $Q_1$, $Q_2$, $Q_3$ of length $2$ 
    such that their interior vertices have degree~$2$ in $G$, 
    then the complement $\overline G$ of $G$ is not a willow.
\end{proposition}

\begin{proof}
    Suppose that $\overline{G}$ is a willow defined by some oriented tree $T$. 
Let $x_1$, $x_2$, $x_3$ be the interior vertices of $Q_1$, $Q_2$, and $Q_3$, respectively. 
    As $\{x_1,x_2,x_3\}$ is a clique in $\overline{G}$, we may assume without loss of generality that $T$ has a directed path $P$ from $x_1$ to $x_3$ whose interior contains $x_2$.
    By \cref{claim:length-2-path}, 
    there is an end $y_2$ of $Q_2$ such that there is no directed path between $x_2$ and $y_2$ in~$T$.

    Since $x_1y_2\in E(\overline{G})$, there exists a directed path $R_1$ in~$T$ between $x_1$ and $y_2$.
    There is no directed path from~$y_2$ to $x_2$ in~$T$ and therefore $R_1$ is directed from~$x_1$ to~$y_2$.
    Similarly, there is a directed path~$R_2$ in~$T$ from~$y_2$ to~$x_3$. Let $R=R_1\cup R_2$. 
    Then, both $P$ and $R$ are directed paths of~$T$ from $x_1$ to $x_3$.
    Since $T$ is a tree, we deduce that $P=R$, 
    contradicting the assumption that there is no directed path between $x_2$ and~$y_2$.
\end{proof}

\begin{figure}
    \centering
    \begin{tikzpicture}[xscale=.8,
        every edge quotes/.style = {auto, font=\footnotesize , sloped
        }, 
        every edge/.style={->,dashed,-{Latex[length=2mm]},draw}]
        \tikzstyle{v}=[circle, draw, solid, fill=black, inner sep=0pt, minimum width=3pt]

        \node at (0,0)[v,label=$v_1$] (v1) {};
        \node at (2,0)[v] (v){};
        \node at (4,0)[v,label=$v_2$] (v2) {};
        \node at (6,0)[v,label=$v_4$] (v3) {};
        \node at (1,1)[v,label=left:$v_5$] (t) {};
        \node at (3,1)[v,label=right:$v_3$] (b) {};
        \draw (v1) edge["$2$"](v);
        \draw (v) edge["$n-2$"] (v2);
        \draw (v2) edge["$1$"] (v3);
        \draw (t) edge["$1$"] (v); 
        \draw (v) edge["$1$"](b);

        \begin{scope}[xshift=10cm]
            \node at (-2,0)[v,label=$v_3$] (v0) {};
            \node at (0,0)[v,label=$v_5$] (v1) {};
            \node at (2,0)[v] (v){};
            \node at (4,0)[v,label=$v_2$] (v2) {};
            \node at (6,0)[v,label=$v_6$] (v3) {};
            \node at (1,1)[v,label=left:$v_4$] (t) {};
            \node at (3,1)[v,label=right:$v_1$] (b) {};
            \draw (v0) edge["$1$"] (v1);
            \draw (v1) edge["$1$"](v);
            \draw (v) edge["$n-2$"] (v2);
            \draw (v2) edge["$1$"] (v3);
            \draw (t) edge["$3$"] (v); 
            \draw (v) edge["$1$"](b);
        \end{scope}
    \end{tikzpicture}
    \caption{Both $\overline{C_5}$ and $\overline{C_6}$ are $n$-willows for every integer $n\ge 5$. Vertices $v_1$, $v_2$, $\ldots$ represent vertices of the antihole in the cyclic order.
    The dashed arc with an integer $k$ means a directed path of length $k$.
}
    \label{fig:C6}
\end{figure}
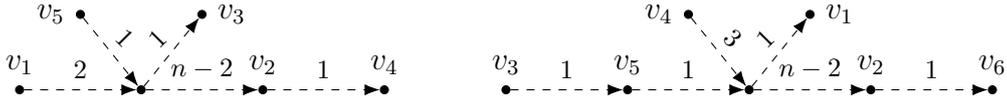
The previous proposition also shows that $\overline{C_n}$ is not a willow for $n\ge 9$. 
It is easy to see that both $\overline{C_5}$ and $\overline{C_6}$ are willows, see \cref{fig:C6}. 
Lastly, we prove that neither $\overline{C_7}$ nor $\overline{C_8}$ is a willow.
We remark that all cycles are willows, see \cref{fig:cycles}.
\begin{proposition}\label{C7}
    The complement $\overline{C_n}$ of $C_n$ is not a willow for all integers $n\ge 7$.
\end{proposition}
\begin{proof}
    Let $v_1$, $v_2$, $\ldots$, $v_n$ be the vertices of $\overline{C_n}$ in cyclic order.
    Suppose that $\overline{C_n}$ is a willow defined by some oriented tree $T$. 
    Let $F$ be the set of all edges $uv$ of $G$ such that there is a directed path from~$u$ to~$v$ or from~$v$ to~$u$ in~$T$.
    
    Suppose that $F=\emptyset$.
    Then for some $j \in \{1,2, \ldots, n \}$, there is no directed path from $v_j$ to~$v_i$ in $T$ for all $i\in\{1, 2,3,\ldots,n\} \setminus \{j\}$. 
    By symmetry, we may assume that $j =1$.
    
    Since $\{v_1,v_3,v_{6}\}$ is a clique of $G$, there is a directed path $P$ in $T$ containing all of $v_1$, $v_3$, and~$v_6$. 
    Let $(i,j,k)$ be the permutation of $\{1,3,6\}$ such that $P$ contains $v_i, v_j, v_k$ in order.
    Then $i=1$ by the assumption on $v_1$.
    Let $\ell\in \{j-1,j+1\}\cap \{4,5\}$.
    Then $\{v_1,v_\ell,v_k\}$ is a clique in~$G$
    and therefore there is a path~$Q$ containing $v_1$, $v_\ell$, and $v_k$. 
    Since $T$ is a tree, $v_j$ is in $V(Q)$, contradicting the assumption that $v_j v_\ell \notin F$.

    Therefore $F\neq\emptyset$. By symmetry, we may assume that $v_2 v_3\in F$.
    Since $T$ contains directed paths between $v_2$ and $v_6$ and between $v_2$ and $v_3$, it follows that
    $T$ contains a directed path $P$ containing $v_2$, $v_3$, and $v_6$.
    Let $(i,j,k)$ be a permutation of $\{2,3,6\}$ such that 
    $P$ is a directed path containing $v_i, v_j, v_k$, in order. 
    By \cref{claim:length-2-path}, $v_{j-1}v_j\notin F$ or $v_jv_{j+1}\notin F$.
    Thus, there is an $\ell\in \{j-1,j+1\}\cap \{1,4,5,7\}$ such that $v_\ell v_j\notin F$.
    Since $v_\ell$ is complete to $\{v_i,v_k\}$, 
    there is a directed path~$Q$ of~$T$ containing $v_i$, $v_k$, and $v_\ell$. 
    As $T$ is a tree, we conclude that $Q$ contains $P$ and therefore $v_j$, contradicting the assumption that $v_j v_\ell\notin F$.
\end{proof}

\begin{figure}
    \centering
    \begin{tikzpicture}[xscale=.7,yscale=.8,
        every edge quotes/.style = {auto, font=\footnotesize , sloped
        }, 
        every edge/.style={->,dashed,-{Latex[length=2mm]},draw}]
        \tikzstyle{v}=[circle, draw, solid, fill=black, inner sep=0pt, minimum width=3pt]

        \node at (-1,-1) [v,label=below:$v_2$] (v0-1){};
        \node at (1,-1) [v,label=below:$v_{18}$] (v0-2){};
        \node at (3,-1) [v,label=below:$v_4$] (v1-1){};
        \node at (5,-1) [v,label=below:$v_{16}$] (v1-2){};
        \node at (7,-1) [v,label=below:$v_6$] (v2-1){}; 
        \node at (9,-1) [v,label=below:$v_{14}$] (v2-2){};
        \node at (11,-1) [v,label=below:$v_8$] (v3-1){};
        \node at (13,-1) [v,label=below:$v_{12}$] (v3-2){};
        \node at (1,2.5) [v,label=$v_3$] (w0-1){};
        \node at (3,2.5) [v,label=$v_{17}$] (w0-2){};
        \node at (5,2.5) [v,label=$v_5$] (w1-1){};
        \node at (7,2.5) [v,label=$v_{15}$] (w1-2){};
        \node at (9,2.5) [v,label=$v_7$] (w2-1){}; 
        \node at (11,2.5) [v,label=$v_{13}$] (w2-2){};
        \node at (13,2.5) [v,label=$v_9$] (w3-1){};
        \node at (15,2.5) [v,label=$v_{11}$] (w3-2){};
        \foreach \x in {0,1,2,3} {
            \node at (4*\x,0) [v] (y\x) {};
            \node at (4*\x+2,1.5) [v] (x\x){};
            \draw (x\x) edge["$n-3$"] (y\x);
            \draw (y\x) edge["$1$"] (v\x-1);
            \draw (y\x) edge["$2$"] (v\x-2);
            \draw (w\x-1) edge["$1$"] (x\x);
            \draw (w\x-2) edge["$2$"] (x\x);
        }
        \draw (x0) edge["$n-3$"] (y1);
        \draw (x1) edge["$n-3$"] (y2);
        \draw (x2) edge["$n-3$"] (y3);
        \node at (-1,1) [v,label=$v_1$] (v1) {};
        \draw (v1) edge["$1$"] (y0);
        \node at (15,.5) [v,label=below:$v_{10}$] (v10){};
        \draw (x3) edge["$1$"] (v10);
        
        \begin{scope}[yshift=-5.5cm] 

            \node at (-1,-1) [v,label=below:$v_2$] (v0-1){};
            \node at (1,-1) [v,label=below:$v_{19}$] (v0-2){};
            \node at (3,-1) [v,label=below:$v_4$] (v1-1){};
            \node at (5,-1) [v,label=below:$v_{17}$] (v1-2){};
            \node at (7,-1) [v,label=below:$v_6$] (v2-1){}; 
            \node at (9,-1) [v,label=below:$v_{15}$] (v2-2){};
            \node at (11,-1) [v,label=below:$v_8$] (v3-1){};
            \node at (13,-1) [v,label=below:$v_{13}$] (v3-2){};
            \node at (1,2.5) [v,label=$v_3$] (w0-1){};
            \node at (3,2.5) [v,label=$v_{18}$] (w0-2){};
            \node at (5,2.5) [v,label=$v_5$] (w1-1){};
            \node at (7,2.5) [v,label=$v_{16}$] (w1-2){};
            \node at (9,2.5) [v,label=$v_7$] (w2-1){}; 
            \node at (11,2.5) [v,label=$v_{14}$] (w2-2){};
            \node at (13,2.5) [v,label=$v_9$] (w3-1){};
            \node at (15,2.5) [v,label=$v_{12}$] (w3-2){};
            \foreach \x in {0,1,2,3} {
                \node at (4*\x,0) [v] (y\x) {};
                \node at (4*\x+2,1.5) [v] (x\x){};
                \draw (x\x) edge["$n-3$"] (y\x);
                \draw (y\x) edge["$1$"] (v\x-1);
                \draw (y\x) edge["$2$"] (v\x-2);
                \draw (w\x-1) edge["$1$"] (x\x);
                \draw (w\x-2) edge["$2$"] (x\x);
            }
            \draw (x0) edge["$n-3$"] (y1);
            \draw (x1) edge["$n-3$"] (y2);
            \draw (x2) edge["$n-3$"] (y3);
            \node at (-1,1) [v,label=$v_1$] (v1) {};
            \draw (v1) edge["$1$"] (y0);
            \node at (16,0) [v,label=right:$v_{10}$] (v10){}; 

            \node at (16,-1) [v,label=right:$v_{11}$] (v11){};
            \draw (x3) edge["$n-2$"] (v10);
            \draw (v10) edge ["$1$"] (v11);
        \end{scope}

    \end{tikzpicture}
    \caption{
        These oriented trees certify that cycles of length $18$ and $19$ are  $n$-willows for every integer $n\ge 4$ and can be easily modified to show that     
        all cycles are $n$-willows. 
        Vertices $v_1$, $v_2$, $\ldots$ represent vertices in the cyclic order.
        The dashed arc with an integer $k$ means a directed path of length $k$.}\label{fig:cycles}
\end{figure}
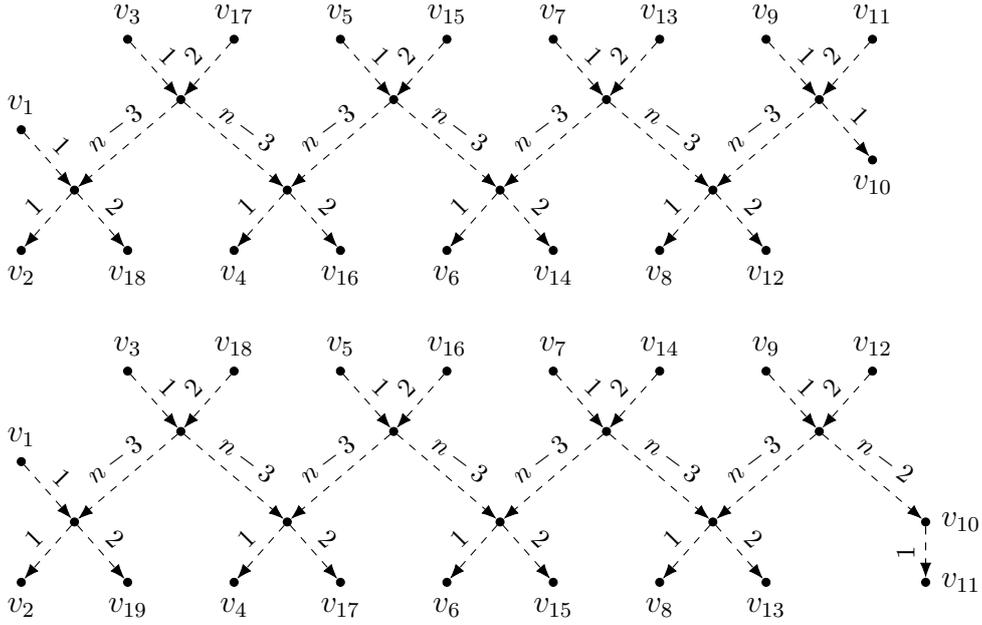

Now we are going to prove that large enough \say{fans} and \say{complete wheels} are not willows.
We define fans as follows. 
Let $n\ge 3$ be an integer. Let $F_n$ be the $(n+1)$-vertex graph with a specified vertex $c$ called the \emph{center} such that $F_n\setminus c$ is the path $P_n$.
A \emph{complete wheel} on $(n+1)$ vertices is the graph $W_n$ obtained from $F_n$ by adding an edge between the two degree\nobreakdash-$1$ vertices of $F_n \setminus c$. 
Hence, $W_n \setminus c$ is the cycle $C_n$.
We will show that $W_n$ and $F_n$ are not willows for each $n \geq 7$.
First, we present a useful lemma.

\begin{lemma}\label{lem:fan-lemma}
    Let $G$ be a copy of $F_4$ with center $c$.
    Let $v_1$ be a vertex of degree one in $G \setminus c$. 
If $G$ is a willow defined by an oriented tree~$T$ 
and $T$ has a directed path from $v$ to $c$ for every $v\in V(G \setminus c)$, 
    then the directed path from $v_1$ to $c$ in $T$ contains at least one vertex in $V(G)\setminus\{v_1,c\}$. 
\end{lemma}
\begin{proof}
    Note $G \setminus c = P_4$.
    Let $v_1$, $v_2$, $v_3$, $v_4$ be the vertices of $P_4$, in order. 
    For each $i\in\{1,2,3,4\}$, let $R_i$ denote the directed path from $v_i$ to $c$ in $T$.
    We may assume that 
    \begin{equation}\label{eq:0}
    V(R_j) \not \subseteq V(R_1) \text{ for each }j \in \{2,3,4\}. 
    \end{equation}
    Since $\{v_1, v_2, c\}$ is a clique there is a directed path $P$ of $T$ containing $v_1, v_2, c$.
    Since $T$ is a tree, $R_1 \cup R_2 = P$.
    Hence, $V(R_1)\subseteq V(R_2)$.
    For $i \in \{2,4\}$, the set $\{v_i,v_3,c\}$ is a clique. Hence,
    \begin{equation}\label{eq:1}
        \text{For every } i \in \{2,4\}, \ V(R_i)\subseteq V(R_3) \text{ or }V(R_i)\subseteq V(R_2).
    \end{equation}
Since $G[\{v_1,v_2,v_4\}]$ is isomorphic to $K_2\cup K_1$, by \cref{obs:multipartite}, 
    \begin{equation}\label{eq:3}
        V(R_4)\not\subseteq V(R_2)\text{ and }V(R_2)\not\subseteq V(R_4).
    \end{equation}

    Suppose that $V(R_2)\subseteq V(R_3)$. 
    By \eqref{eq:1} and \eqref{eq:3}, $V(R_4)\subseteq V(R_3)$
    and therefore $V(R_3)$ contains both $V(R_1)$ and $V(R_4)$. This means that $R_3$ contains $v_1$, $v_3$, $v_4$, contradicting \cref{obs:multipartite}.

    Thus, $V(R_3)\subseteq V(R_2)$.
    Since $V(R_1) \subseteq V(R_2)$ and $R_1, R_2, R_3$ are all directed paths ending at $c$, it follows from \eqref{eq:0} that
    $V(R_1)\subseteq V(R_3)\subseteq V(R_2)$.
    By \eqref{eq:1} and \eqref{eq:3}, $V(R_3)\subseteq V(R_4)$.
    So $R_4$ is a directed path containing each of $v_1, v_3, v_4$ contrary to \cref{obs:multipartite}.
\end{proof}

\begin{figure}
    \centering
    \begin{tikzpicture}[yscale=.8,
        every edge quotes/.style = {auto, font=\footnotesize , sloped
        }, 
        every edge/.style={->,dashed,-{Latex[length=2mm]},draw}]
        \tikzstyle{v}=[circle, draw, solid, fill=black, inner sep=0pt, minimum width=3pt]

        \node at(0,0)[v,label=below:$c$] (c){};
        \node at (0,1)[v,label=$v_1$] (v1){};
        \node at (-2,1)[v,label=$v_3$] (v3){};
        \node at (-3,1.5)[v,label=below:$v_4$] (v4){};
        \node at (2,1)[v,label=$v_6$] (v6){};
        \node at (-1,-1)[v,label=below:$v_2$] (v2){};
        \node at (2,-1)[v,label=below:$v_5$] (v5){};
        \draw (v1)edge["$2$"] (c);
        \draw (v3)edge["$2$"] (c);
        \draw (v6)edge["$3$"] (c);
        \draw (v4)edge["$1$"] (v3);
        \draw (c)edge["$2$"] (v2);
        \draw (c)edge["$3$"] (v5);

        \begin{scope}[xshift=6cm]
                
            \node at(0,0)[v,label=below:$c$] (c){};
            \node at (-2,1.5)[v,label=$v_1$] (v1){};
            \node at (-1,1)[v,label=$v_2$] (v2){};
            \node at (-1,-1)[v,label=below:$v_3$] (v3){};
            \node at (1,1)[v,label=$v_4$] (v4){};
            \node at (2,1.5)[v,label=below:$v_5$] (v5){};
            \node at (1,-1)[v,label=below:$v_6$] (v6){};
            \draw (v1)edge["$1$"] (v2);
            \draw (v2)edge["$2$"] (c);
            \draw (c)edge["$2$"] (v3);
            \draw (c)edge["$3$"] (v6);
            \draw (v5)edge["$1$"] (v4);
            \draw (v4)edge["$2$"] (c);
        \end{scope}
    \end{tikzpicture}
    \caption{
        Both $F_6$ and $W_6$ are $5$-willows. 
        Vertices $v_1$, $v_2$, $\ldots$ represent vertices in the  order in $F_6\setminus c$ or $W_6\setminus c$. 
        The dashed arc with an integer $k$ means a directed path of length $k$.}\label{fig:f6}
\end{figure}
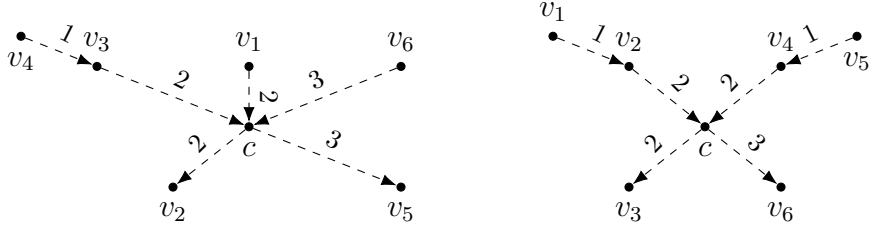
Note that $F_6$ is a willow, see \cref{fig:f6}.
We prove that $F_n$ is not a willow if $n\ge7$.
\begin{proposition}\label{prop:fan}
    For every integer $n\ge 7$, $F_n$ is not a willow.
\end{proposition}
\begin{proof}
    Let $G:=F_n$. 
    Suppose that $G$ is an $m$-willow defined by an oriented tree $T$
    for a positive integer $m$.
    Let $A$ be the vertices of $G$ from which $T$ has a directed path to $c$.
    Let $B$ be the vertices of $G$ to which $T$ has a directed path from $c$.
    Since $c$ is complete to $V(G)\setminus \{c\}$, 
    $A\cup B=V(G)\setminus \{c\}$. 
Let $v_1$, $v_2$, $\ldots$, $v_n$ be the vertices of $G\setminus c$ in the order defined by the path~$G \setminus c$.

    \begin{claim}\label{claim:fan-independent}
        Either $A$ is an independent set of $G$ or $B$ is empty.
    \end{claim} 
    \begin{subproof}
        Suppose that $A$ contains an edge $v_iv_{i+1}$.
        There is a directed path of $T$ from $v_i$ or $v_{i+1}$ to~$c$ containing all of $v_i$, $v_{i+1}$, and $c$.
        Let $M = (N_G(x)\cup N_G(y))\setminus\{c\}$.
        Then by definition, $M$ contains at most two vertices of $G \setminus c$, namely $v_{i-1}$ if $i > 1$ and $v_{i+2}$ if $i < n$.
        Let $X=V(G)\setminus (M \cup \{c\}) $.
        For each vertex $z\in X$, 
        $G[\{x,y,z\}]$ induces a graph isomorphic to $K_2\cup K_1$
        and therefore $z\notin B$ by \cref{obs:multipartite}.
        So, $X\subseteq A$.
        Since $n\ge 7$, $v_1,v_2\in X$ or $v_{n-1},v_n\in X$.
        We deduce that $\{v_1,v_2,v_{n-1},v_n\}\subseteq A$ 
        by \cref{obs:multipartite}
        because each of its $3$-vertex subsets induces a subgraph of $G$ isomorphic to $K_2\cup K_1$. 
        For every vertex $w\in V(G)\setminus(X\cup \{c\})$, there are distinct vertices $u,v\in \{v_1,v_2,v_{n-1},v_n\}$ 
        such that $uv$ is an edge of $G$ and $w$ is non-adjacent to both $u$ and $v$. Again by \cref{obs:multipartite}, $w\in A$.
        Hence, $B = \emptyset$.
    \end{subproof}

    Suppose that $B=\emptyset$. 
    Choose a vertex $v$ in $A$ such that $d_T(v,c)$ is minimized. 
    Then $G\setminus c$ has a $4$-vertex induced path  
    starting at $v$ because $n\ge 7$. 
    By \cref{lem:fan-lemma}, the directed path from $v$ to $c$ contains at least one vertex of $V(G)\setminus\{c,v\}$, contradicting the choice of~$v$. 
    Therefore we may assume that $B\neq\emptyset$. 
    By symmetry, $A\neq\emptyset$.
    By \cref{claim:fan-independent}, both $A$ and $B$ are independent sets of~$G$.
    
    We may assume that $A$ contains $v_i$ for each even $i \in \{1,2, \dots, n\}$ and $B$ contains $v_j$  for every odd $j \in \{1,2, \dots, n \}$.
For each $i\in\{1,2,\ldots,n-5\}$, 
    $d_T(v_i,c)\equiv d_T(v_{i+2},c)\pmod m$ 
    because $v_{i+5}$ is non-adjacent to both $v_i$ and $v_{i+2}$.
    Similarly, for each $i\in \{6,7, \ldots,n\}$, 
    $d_T(v_{i-2},c)\equiv d_T(v_{i},c)\pmod m$ 
    because $v_{i-5}$ is non-adjacent to both $v_i$ and $v_{i-2}$.

    So, there are integers $a$ and $b$ such that $d_T(v_i,c)\equiv a\pmod m$ for all even $i\in\{1,2,\ldots,n\}$ 
    and $d_T(v_i,c)\equiv b\pmod m$ for all odd $i\in\{1,2,\ldots,n\}$.  
    This implies that $A$ is complete or anti-complete to $B$, a contradiction.
\end{proof}

Since $F_n$ is an induced subgraph of $W_{n+1}$, 
by \cref{prop:fan}, $W_n$ is not a willow for all $n\ge 8$. 
However, it is easy to see that $W_n$ is a willow for every $n<7$, see \cref{fig:f6}.
We now show that $W_7$ is not a willow. 

\begin{proposition}\label{prop:wheel}
    For every integer $n\ge 7$, $W_n$ is not a willow.
\end{proposition}
\begin{proof}
    Let $G:=W_n$. 
    Suppose that $G$ is an $m$-willow defined by an oriented tree $T$ 
    for a positive integer $m$.
    Let $A$ be the vertices of $G$ from which $T$ has a directed path to $c$.
    Let $B$ be the vertices of $G$ to which $T$ has a directed path from $c$.
    Since $c$ is complete to $V(G)\setminus \{c\}$, 
    $A\cup B=V(G)\setminus \{c\}$. 

    \begin{claim}\label{claim:wheel-independent}
        Either $A$ is an independent set of $G$ or $B$ is empty.
    \end{claim} 
    \begin{subproof}
        Suppose that $A$ contains an edge $xy$.
        There is a directed path of $T$ from $x$ or $y$ to~$c$ containing all of $x$, $y$, and $c$.
        Let $X=V(G)\setminus (N_G(x)\cup N_G(y)\cup\{c\})$.
        For each vertex $z\in X$, 
        $G[\{x,y,z\}]$ induces a graph isomorphic to $K_2\cup K_1$
        and therefore $z\notin B$ by \cref{obs:multipartite}.
        Since $n\ge 7$, $\abs{X}\ge 3$ and $X\subseteq A$.
        Then for every vertex $w\in V(G)\setminus(X\cup \{c\})$, there are distinct vertices $u,v\in X$
        such that $uv$ is an edge of $G$ and $w$ is non-adjacent to both $u$ and $v$. Again by \cref{obs:multipartite}, $w\in A$. Hence, $B=\emptyset$.
    \end{subproof}

    Suppose that $B=\emptyset$. 
    Choose a vertex $v$ in $A$ such that $d_T(v,c)$ is minimized. 
    By \cref{lem:fan-lemma}, the directed path from $v$ to $c$ contains at least one vertex of $V(G)\setminus\{c,v\}$, contradicting the choice of~$v$. 
    Therefore we may assume that $B\neq\emptyset$. 
    By symmetry, $A\neq\emptyset$.
    By \cref{claim:wheel-independent}, both $A$ and $B$ are independent sets of $G$, so $n$ is even.

    Let $v_1$, $v_2$, $\ldots$, $v_n$ be the vertices of $G\setminus c$ in the cyclic order. We assume that $v_{n+k}=v_k$ for all $k\in\{1,2,\ldots,n\}$.
    We may assume that $v_1,v_3,\ldots,v_{n-1}\in A$
    and $v_2,v_4,\ldots,v_n\in B$ by swapping $A$ and $B$ if necessary.
    For each $i\in\{2,4,\ldots,n\}$, $d_T(v_i,c)\equiv d_T(v_{i+2},c)\pmod m$ 
    because $v_{i+5}\in A$ is non-adjacent to both $v_i$ and $v_{i+2}$.
    So, there is an integer $a$ such that $d_T(v_i,c)\equiv a\pmod m$ for all $i\in\{2,4,\ldots,n\}$.
    Similarly, there is an integer $b$ such that $d_T(c,v_j)\equiv b\pmod m$ for all $j\in\{1,3,\ldots,n-1\}$.
    This implies that $A$ is complete or anti-complete to $B$, a contradiction.
\end{proof}

Now \cref{thm:nonpollyana} follows from \cref{non-pollyanna} and the propositions in this section.

\section{Further work}
\label{sec:conclusions}
We believe that Pollyanna classes of graphs provide a fruitful framework to study the structural distinctions between polynomially $\chi$-bounded classes and $\chi$-bounded classes that are not polynomially $\chi$-bounded.
We conclude our paper by outlining some open problems.

We remark that every Pollyanna graph class discussed in this paper is also strongly Pollyanna, which begs the following question:
\begin{problem}\label{problem:strong-pollyanna}
    Are there Pollyanna graph classes that are not strongly Pollyanna?
\end{problem}
Resolving Problem \ref{problem:strong-pollyanna} would likely require a better understanding of $k$-good graph classes which are not $\chi$-bounded, which have only recently been proven to exist \cite{CARBONERO202363}.
\cref{non-nearchi} gives more examples of $k$-good graph classes which are not $\chi$-bounded.

In a recent paper, Bourneuf and Thomassé \cite{bourneuf2023bounded} introduce an operation called \say{delayed-extension} which preserves polynomial $\chi$-boundedness on a class of graphs. We comment that the delayed-extension of a (strongly) Pollyanna class is also (strongly) Pollyanna, which gives us a slight improvement of \cref{thm:technical main}. 
In \cite{bourneuf2023bounded}, Bourneuf and Thomassé suggest that better understanding the classes which can be obtained from simple graph classes by applying delayed-extension a finite number of times should be helpful in understanding (polynomial) $\chi$-boundedness.
We also point out that this may be a good approach to better understanding Pollyanna graph classes. 

A \emph{wheel} is a graph consisting of an induced cycle of length at least four and a single additional vertex with at least three neighbors on the cycle.
The class of graphs with no induced wheel is not $\chi$-bounded~\cite{davieswheel,pournajafi2020burling,POURNAJAFI2024103849}, however, it may well be Pollyanna.
The fact that the class of (wheel,theta)-free graphs is linearly $\chi$-bounded~\cite{wheelMR4089570} provides some limited evidence that the class of wheel-free graphs might be Pollyanna.
We remark that we showed in \cref{prop:wheel} that for every \emph{finite} set $\mathcal{F}$ of complete wheels of length at least seven, the class of $\mathcal{F}$-free graphs is \emph{not} Pollyanna. However, in our opinion this does not provide evidence that the class of wheel-free graphs is not Pollyanna. 
\begin{problem} 
    Is the class of wheel-free graphs Pollyanna?
\end{problem}

We note that even though Esperet's conjecture was disproved, it is still open whether the Gy\'arf\'as-Sumner Conjecture holds in the following stronger sense:
\begin{problem}[Polynomial Gy\'arf\'as-Sumner]\label{prob:poly-gs}
    Is it true that for every forest $F$ the class of $F$-free graphs is \emph{polynomially} $\chi$-bounded?
\end{problem}

We say a graph $H$ is \emph{Pollyanna-binding} if the class of $H$-free graphs is Pollyanna.
In this language, we may ask if every forest is Pollyanna-binding.
An even more ambitious open problem is to characterize the class of Pollyanna-binding graphs. While we gave some results in this direction, we are quite far from a full characterization. 
We ask about some special cases we believe may be more tractable.

\begin{figure}
    \centering
    \begin{subfigure}{.30\textwidth}
        \centering
        \begin{tikzpicture}
            \tikzstyle{v}=[circle, draw, solid, fill=black, inner sep=0pt, minimum width=3pt]
            \foreach \x in {0,1,2,3,4,5,6} 
            { 
                \node [v] at (360*\x/7:1) (v\x) {};
            }
            \begin{scope}[xshift=2cm]
                \foreach \x in {0,1,2,3,4}
                { 
                    \node [v] at (360*\x/5+180:1) (w\x) {};
                }
            \end{scope}
            \foreach \x in {0,1,2,3,4,5}
            {
                \pgfmathtruncatemacro{\lb}{\x+1}
                \foreach \y in {\lb,...,6}
                {
                    \draw (v\x)--(v\y);
                }
            }
            \foreach \x in {0,1,2,3}
            {
                \pgfmathtruncatemacro{\lb}{\x+1}
                \foreach \y in {\lb,...,4}
                {
                    \draw (w\x)--(w\y);
                }
            }
        \end{tikzpicture}
        \caption{A $(6,4)$-bowtie.}\label{fig:st-bowtie}
    \end{subfigure}
    \begin{subfigure}{.30\textwidth}
        \centering
        \begin{tikzpicture}
            \tikzstyle{v}=[circle, draw, solid, fill=black, inner sep=0pt, minimum width=3pt]
            \foreach \x in {0,1,2,3,4,5,6} 
            { 
                \node [v] at (360*\x/7:1) (v\x) {};
            }
            \begin{scope}[xshift=2.5cm]
                \foreach \x in {0,1,2,3,4}
                { 
                    \node [v] at (360*\x/5+180:1) (w\x) {};
                }
            \end{scope}
            \foreach \x in {0,1,2,3,4,5}
            {
                \pgfmathtruncatemacro{\lb}{\x+1}
                \foreach \y in {\lb,...,6}
                {
                    \draw (v\x)--(v\y);
                }
            }
            \foreach \x in {0,1,2,3}
            {
                \pgfmathtruncatemacro{\lb}{\x+1}
                \foreach \y in {\lb,...,4}
                {
                    \draw (w\x)--(w\y);
                }
            }
            \draw (v0)--(w0);
        \end{tikzpicture}
        \caption{A $(7,5)$-dumbbell.}\label{fig:dumbbell}
    \end{subfigure}
    \begin{subfigure}{.30\textwidth}
        \centering
        \begin{tikzpicture}
            \tikzstyle{v}=[circle, draw, solid, fill=black, inner sep=0pt, minimum width=3pt]
            \node [v] at (0,0.5) (v0){};
            \node [v] at (60:1) (v1){};
            \node [v] at (120:1) (v2){};
            \draw (v1)--+(0,-1) node [v]{};
            \draw (v2)--+(0,-1) node [v]{};
            \draw (v0)--+(0,-1) node [v]{};
            \draw (v0)--(v1)--(v2)--(v0);
        \end{tikzpicture}
        \caption{A tripod.}\label{fig:tripod}
    \end{subfigure}
    \caption{Graphs appearing in the problems.}\label{fig:smallgraph2}
\end{figure}
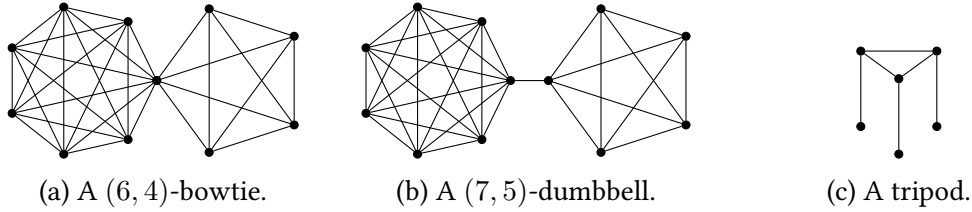

We call a graph an \emph{$(s,t)$-bowtie} if it can be obtained from the disjoint union of $K_s$ and $K_t$ by adding a new vertex complete to everything else, see \cref{fig:st-bowtie}.
In this language, \cref{bowtie:main} states that the $(2,2)$-bowtie is Pollyanna-binding.
\begin{problem}
    Is the class of $(s,t)$-bowtie-free graphs Pollyanna for each $s \geq 3$ and $t \geq 2$? 
\end{problem}
We call a graph an \emph{$(s,t)$-dumbbell} if it can be obtained from the disjoint union of $K_s$ and~$K_t$ by adding a single additional edge between a vertex of the $K_s$ and a vertex of the $K_t$, see \cref{fig:dumbbell}.
Note that a $t$-lollipop is a $(2,t)$-dumbbell, so \cref{lollipop-free2} states that the class of $(2,t)$-dumbbell-free graphs is Pollyanna.
\begin{problem}
    Is the class of $(s,t)$-dumbbell-free graphs Pollyanna for each $s \geq 3$ and $t \geq 3$? 
\end{problem}

Bulls are induced subgraphs of certain pentagram spiders. While the class of bull-free graphs is Pollyanna by \cref{bulls-main}, the class of pentagram spider-free graphs is not by \cref{non-pollyanna,pentagram spider}. The next natural case to consider would be tripod-free graphs. A \emph{tripod} is the 
graph obtained from $K_3$ by adding one pendant vertex to each vertex of the $K_3$, see \cref{fig:tripod}.
\begin{problem}
    Is the class of tripod-free graphs Pollyanna?
\end{problem}

A \emph{$K_{k}$-free coloring} of a graph $G$ is a coloring of the vertices of $G$ such that no monochromatic subgraph of $G$ contains a copy of $K_k$.
The \emph{$K_k$-free chromatic number} $\chi_{k}(G)$ of a graph~$G$~\cite{SS1968,BF1990,KW2017,KP2018} is the minimum $k$ such that $G$ has a $K_k$-free coloring.
Thus $\chi(G)=\chi_2(G)\ge\chi_3(G)\ge \chi_4(G)\ge\cdots\ge \chi_{\omega(G)+1}=1$.
We say that a class~$\C$ of graphs is \emph{polynomially $\chi_k$-bounded} if there is a polynomial~$f$ such that for every induced subgraph~$H$ of a graph~$G\in \C$, we have $\chi_k(G)\le f(\omega(G))$.
It is easy to see that every polynomially $\chi_{k}$-bounded class of graphs is $(k-1)$-strongly Pollyanna.
In fact, \cref{thm:technical main}\ref{item:bowtie} is shown by proving that the class of bowtie-free graphs is polynomially $\chi_4$-bounded.

We now aim to present one class of graphs defined by geometric representations and show that the argument of Krawczyk and Walczak~\cite{KW2017} implies it is Pollyanna. 
Let $\mathcal F$ be a finite family of sets. 
We say $A$ and $B$ \emph{overlap} if $A\cap B\neq \emptyset$ and neither $A$ nor $B$ is a subset of the other.
The \emph{overlap graph} of a family $\mathcal F$ of finite sets is a graph on $\mathcal F$ such that two sets in $\mathcal F$ are adjacent if and only if they overlap.
A rectangle on the plane is a set of the form $I\times J$ where $I$ and $J$ are closed intervals on the real line.
The \emph{clean directed rectangle overlap graphs} are the overlap graphs of a finite family $\mathcal F$ of axis-aligned rectangles in the plane such that 
\begin{enumerate}
    \item (clean) no member of $\mathcal F$ is completely contained in the intersection of two other ovelapping members of $\mathcal F$, 
    \item (directed) if two members $I_1\times J_1$ and $I_2\times J_2$ of $\mathcal F$ overlap, then either $J_2\subseteq J_1$ and $\min(I_1)<\min(I_2)$
    or $J_1\subseteq J_2$ and $\min(I_2)<\min(I_1)$.
\end{enumerate}
Krawczyk, Pawlik, and Walczak~\cite[Lemma 2.3]{KPW2015} showed that every clean directed rectangle overlap graph is a clean interval overlap game graph. (See Krawczyk and Walczak~\cite[Lemma 6.1]{KW2017} which uses slightly different terminology.) Krawczyk and Walczak~\cite[Lemma 7.4]{KW2017} showed that there is an on-line $K_3$-free coloring algorithm that uses at most $O(\omega(G)^3)$ colors for clean interval overlap graphs defined with a natural rule.
Krawczyk and Walczak~\cite[Lemma 2.1]{KW2017} showed that if there is an on-line $K_3$-free coloring algorithm that uses at most $c$ colors for 
a clean interval overlap graph, 
then its game graph has the $K_3$-free chromatic number at most $c$.
By combining these facts, we deduce the following. 
\begin{proposition}\label{prop:clean-rectangle}
    The class of clean directed rectangle overlap graphs is polynomially $\chi_3$-bounded and therefore $2$-strongly Pollyanna.
\end{proposition}
Hence, we wonder if Pollyanna graph classes are a fruitful framework in which to examine non-$\chi$-bounded geometric graph classes. 
We ask the following question, which was suggested by one of the anonymous referees of this paper:
\begin{problem}\label{problem:pollyanna-geometric}
    Are there other graph classes defined by geometric representations that are Pollyanna but not $\chi$-bounded?
    In particular, what about string graphs?
\end{problem}

Scott and Seymour \cite{SCOTT201668} proved that the class of odd hole-free graphs is $\chi$-bounded. Their $\chi$-bounding function is doubly exponential and it remains open whether the class of odd-hole-free graphs is polynomially $\chi$-bounded (and so Pollyanna).
We propose the analogous problem for odd antihole-free graphs.
Note that the class of odd antihole-free graphs is not $\chi$-bounded because there are graphs of girth at least $6$ with arbitrarily large chromatic number, shown by Erd\H{o}s~\cite{Erdos1959}.

\begin{problem}
    Is the class of odd antihole-free graphs Pollyanna?
\end{problem}

\cref{C7} shows that no antihole of length at least $7$ is a willow. However, small antiholes such as $\overline{C_5}$ and $\overline{C_6}$ are. It may well be true that the class of $C_5$-free graphs is Pollyanna.
Antihole-free graphs are polynomially $\chi$-bounded since $\overline{C_4}=2K_2$ \cite{Wagon1980}.
But, the class of graphs without antiholes of length at least $5$ is not $\chi$-bounded, again by the result of Erd\H{o}s~\cite{Erdos1959}.
So, as a starting point, we propose the following problem.

\begin{problem}
    Is the class of graphs without any antihole of length at least 5 Pollyanna?
\end{problem}

The simplest willows are those whose underlying oriented tree is a directed path between two vertices. These graphs are exactly the complete multipartite graphs, thus it is natural to consider if a class of graphs with a forbidden complete multipartite graph is Pollyanna. In this direction, the first step would be to determine whether the class of graphs without an induced square $K_{2,2}=C_4$ or an induced diamond $K_{2,1,1}=K_4 \setminus e$ is Pollyanna.

\begin{problem}
    Is the class of $\{ C_4, K_4\setminus e \}$-free graphs Pollyanna?
\end{problem}

In \cref{sec:willows}, we described some forbidden induced subgraphs for willows but did not have a complete list of forbidden induced subgraphs for willows.
\begin{problem}
    Characterize willows by their minimal forbidden induced subgraphs.
\end{problem}

In \cref{sec:nonpollyanna}, we showed that all Pollyanna-binding graphs are willows.
Based on this, we can end our paper with the following extremely optimistic conjecture.
\begin{conjecture}[Pollyanna's Conjecture]
    A graph is Pollyanna-binding if and only if it is a willow.
\end{conjecture}

If Pollyanna's conjecture is disproved, then Pollyanna \cite{porter1913pollyanna} would almost certainly immediately make a new equally na{\"i}ve conjecture.

\section*{Acknowledgements}
The initial idea for this paper was developed at the MATRIX-IBS Workshop: Structural Graph Theory Downunder III. 
Much of this work was done while James Davies and Maria Chudnovsky were visiting the Institute for Basic Science (IBS).
We are grateful for the generous support from MATRIX and IBS for making this project possible.
We thank Freddie Illingworth, Robert Hickingbotham, and G\"{u}nter Rote for their helpful discussions.
We also thank the anonymous referees for their careful reading and helpful comments. In particular, \cref{prop:clean-rectangle,problem:pollyanna-geometric} are due to the referee's suggestion.

\newcommand{\etalchar}[1]{$^{#1}$}

\end{document}